\documentclass[a4paper]{amsart}
\usepackage{color}
\usepackage[latin1]{inputenc}
\usepackage[T1]{fontenc}
\usepackage{amsmath,amssymb,amsfonts}
\usepackage{amsthm}
\usepackage{stmaryrd}
\usepackage{enumerate}
\newtheorem{theorem}{Theorem}[section]
\newtheorem{lemma}[theorem]{Lemma}
\newtheorem{proposition}[theorem]{Proposition}

\newtheorem{definition}[theorem]{Definition}
\newtheorem{assumption}[theorem]{Assumption}

\begin{document}
\setlength\arraycolsep{2pt}
\title{Sweeping Processes Perturbed by Rough Signals}
\author{Charles CASTAING*}
\address{*D\'epartement de Math\'ematiques, Universit\'e Montpellier 2, Montpellier, France}
\author{Nicolas MARIE**}
\address{**Laboratoire Modal'X, Universit\'e Paris 10, Nanterre, France}
\email{nmarie@parisnanterre.fr}
\address{**ESME Sudria, Paris, France}
\email{nicolas.marie@esme.fr}
\author{Paul RAYNAUD DE FITTE***}
\address{***Laboratoire Rapha\"el Salem, Universit\'e de Rouen Normandie, UMR CNRS 6085, Rouen, France}
\email{prf@univ-rouen.fr}
\date{}
\maketitle
\noindent
%


%
\begin{abstract}
 This paper deals with the existence, the uniqueness and an approximation scheme of the solution to sweeping
 processes perturbed by a continuous signal of finite $p$-variation
 with $p\in [1,3[$. It covers pathwise stochastic
 noises directed by a fractional Brownian motion of Hurst parameter
 greater than $1/3$.
\end{abstract}
\tableofcontents
%


%
\section{Introduction}
Consider a multifunction $C : [0,T]\rightrightarrows\mathbb R^e$ with $e\in\mathbb N^*$. Roughly speaking, the Moreau sweeping process (see Moreau \cite{MOREAU76}) associated to $C$ is the path $X$, living in $C$, such that when it hits the frontier of $C$, a minimal force is applied to $X$ in order to keep it inside of $C$. Precisely, $X$ is the solution to the following differential inclusion:
\begin{equation}\label{sweeping_process}
\left\{
\begin{array}{rcl}
 -\displaystyle{\frac{dDY}{d|DY|}}(t) & \in & N_{C(t)}(Y(t))
 \textrm{ $|DY|$-a.e.}\\
 Y(0) & = & a\in C(0),
\end{array}
\right.
\end{equation}
where $DY$ is the differential measure associated with the continuous
function of bounded variation $Y$, $|DY|$ is its variation measure,
and $N_{C(t)}(Y(t))$ is the normal cone of $C(t)$ at $Y(t)$. This problem has been deeply studied by many authors. For instance, the reader can refer to Moreau \cite{MOREAU76}, Valadier \cite{VALADIER90} or Monteiro Marques \cite{MONTEIRO_MARQUES93}.
\\
\\
Several authors studied some perturbed versions of Problem (\ref{sweeping_process}), in particular by a stochastic multiplicative noise in It\^o's calculus framework (see Revuz and Yor \cite{RY99}). For instance, the reader can refer to Bernicot and Venel \cite{BV11} or Castaing et al.~\cite{CMR15}. On reflected diffusion processes, which perturbed sweeping processes with constant constraint set, the reader can refer to Kang and Ramanan \cite{KR10}.
\\
\\
Consider the perturbed Skorokhod problem
\begin{equation}\label{rough_sweeping_process}
\left\{
\begin{array}{rcl}
 X(t) & = &
 H(t) + Y(t)\\
 H(t) & = &
 \displaystyle{
 \int_{0}^{t}f(X(s))dZ(s)}\\
 -\displaystyle{\frac{dDY}{d|DY|}}(t) & \in & N_{C_H(t)}(Y(t))
 \textrm{ $|DY|$-a.e. with }
 Y(0) = a,
\end{array}
\right.
\end{equation}
where $C_H(t) = C(t)-H(t)$, $t \in [0, T]$ 
(thus $N_{C_H(t)}(Y(t))=N_{C(t)}(X(t))$),  
$Z : [0,T]\rightarrow\mathbb R^d$ is a continuous signal of finite
$p$-variation with $d\in\mathbb N^*$ and $p\in [1,\infty[$,
$f\in\textrm{Lip}^{\gamma}(\mathbb R^e,\mathcal M_{e,d}(\mathbb R))$
with $\gamma > p$, and the integral against $Z$ is taken in the sense
of rough paths. On the rough integral, the reader can refer to Lyons
\cite{LYONS04}, Friz and Victoir \cite{FV10} or Friz and Hairer
\cite{FH14}. Throughout the paper, the multifunction $C$ satisfies the
following assumption. 
%


%
\begin{assumption}\label{assumption_C}
$C$ is a convex compact valued multifunction, continuous for the Hausdorff distance, and there exists a continuous selection $\gamma : [0,T]\rightarrow\mathbb R^e$ satisfying
\begin{displaymath}
\overline B_e(\gamma(t),r)\subset\normalfont{\textrm{int}}(C(t))
\textrm{ $;$ }
\forall t\in [0,T],
\end{displaymath}
where $\overline B_e(\gamma(t),r)$ denotes the closed ball of radius
$r$ centered at $\gamma(t)$. 
\end{assumption}
\noindent
This assumption is equivalent to saying that $C(t)$ has nonempty interior for every $t\in[0,T]$, see \cite[Lemma 2.2]{CMR15}.
\\
\\
In Falkowski and S{\l}omi\'nski \cite{FS15}, when $p\in [1,2[$ and $C(t)$ is a cuboid of $\mathbb R^e$ for every $t\in [0,T]$, the authors proved the existence and uniqueness of the solution of Problem (\ref{rough_sweeping_process}). Furthermore, several authors studied the existence and uniqueness of the solution for reflected rough differential equations. In \cite{AIDA15}, M.~Besal\'u et al.~proved the existence and uniqueness of the solution for delayed rough differential equations with non-negativity constraints. Recently, S.~Aida gets the existence of solutions for a large class of reflected rough differential equations in \cite{AIDA16} and \cite{AIDA15}. Finally, in \cite{DGHT16}, A.~Deya et al.~proved the existence and uniqueness of the solution for 1-dimensional reflected rough differential equations. An interesting remark related to these references is that when $C$ is not a cuboid, moving or not, it is a challenge to get the uniqueness of the solution for reflected rough differential equations and sweeping processes.
\\
\\
For $p\in [1,3[$, the purpose of this paper is to prove the existence of solutions to Problem (\ref{rough_sweeping_process}) when $C$ satisfies Assumption \ref{assumption_C}, and a necessary and sufficient condition for uniqueness close to the monotonicity of the normal cone which allows to prove the uniqueness when $p = 1$ and there is an additive continuous signal of finite $q$-variation with $q\in [1,3[$. In this last case, the convergence of an approximation scheme is also proved.
\\
\\
Section 2 deals with some preliminaries on sweeping processes and the
rough integral. Section 3 is devoted to the existence of solutions to
Problem (\ref{rough_sweeping_process}) when $Z$ is a moderately
irregular signal (i.e. $p\in [1,2[$) and when $Z$ is a rough signal
(i.e. $p\in [2,3[$). Section 4 deals with some uniqueness results. The
convergence of an approximation scheme based on Moreau's catching up algorithm is proved in Section 5 when $p = 1$ and there is an additive continuous signal of finite $q$-variation with $q\in [1,3[$. Finally, Section 6 deals with sweeping processes perturbed by a pathwise stochastic noise directed by a fractional Brownian motion of Hurst parameter greater than $1/3$.
\\
\\
The following notations, definitions and properties are used throughout the paper.
\\
\\
\textbf{Notations and elementary properties:}
\begin{enumerate}[1.]
 \item $C_h(t) := C(t) - h(t)$ for every function $h : [0,T]\rightarrow\mathbb R^e$.
 \item $N_{C}(x)$ is the normal cone of $C$ at $x$, for any closed convex subset $C$ of $\mathbb R^e$ and any $x\in\mathbb R^e$ (recall that $N_{C}(x)=\emptyset$ if $x\not\in C$).
 \item $\Delta_T :=\{(s,t)\in [0,T]^2 : s < t\}$ and $\Delta_{s,t} :=\{(u,v)\in [s,t]^2 : u < v\}$ for every $(s,t)\in\Delta_T$.
 \item For every function $x$ from $[0,T]$ into $\mathbb R^d$ and $(s,t)\in\Delta_T$, $x(s,t) := x(t) - x(s)$.
 \item Consider $(s,t)\in\Delta_T$. The vector space of continuous functions from $[s,t]$ into $\mathbb R^d$ is denoted by $C^0([s,t],\mathbb R^d)$ and equipped with the uniform norm $\|.\|_{\infty,s,t}$ defined by
 \begin{displaymath}
 \|x\|_{\infty,s,t} :=
 \sup_{u\in [s,t]}\|x(u)\|
 \end{displaymath}
 for every $x\in C^0([s,t],\mathbb R^d)$, or the semi-norm $\|.\|_{0,s,t}$ defined by
 \begin{displaymath}
 \|x\|_{0,s,t} :=
 \sup_{u,v\in [s,t]}\|x(v) - x(u)\|
 \end{displaymath}
 for every $x\in C^0([s,t],\mathbb R^d)$. Moreover, $\|.\|_{\infty,T} :=\|.\|_{\infty,0,T}$, $\|.\|_{0,T} := \|.\|_{0,0,T}$ and
 \begin{displaymath}
 C_{0}^{0}([s,t],\mathbb R^d) :=
 \{x\in C^0([s,t],\mathbb R^d) : x(0) = 0\}.
 \end{displaymath}
 \item Consider $(s,t)\in\Delta_T$. The set of all dissections of $[s,t]$ is denoted by $\mathfrak D_{[s,t]}$ and the set of all strictly increasing sequences $(s_n)_{n\in\mathbb N}$ of $[s,t]$ such that $s_0 = s$ and $\lim_{\infty} s_n = t$ is a denoted by $\mathfrak D_{\infty,[s,t]}$.
 \item Consider $(s,t)\in\Delta_T$. A function $x : [s,t]\rightarrow\mathbb R^d$ has finite $p$-variation if and only if,
 \begin{eqnarray*}
  \|x\|_{p\textrm{-var},s,t} & := &
  \sup\left\{\left|
  \sum_{k = 1}^{n - 1}\|x(t_k,t_{k + 1})\|^p\right|^{1/p}\textrm{$;$ }
  n\in\mathbb N^*\textrm{ and }
  (t_k)_{k\in\llbracket 1,n\rrbracket}\in\mathfrak D_{[s,t]}\right\}\\
  & < & \infty.
 \end{eqnarray*}
 Consider the vector space
 \begin{displaymath}
 C^{p\textrm{-var}}([s,t],\mathbb R^d) :=
 \{x\in C^0([s,t],\mathbb R^d) :\|x\|_{p\textrm{-var},s,t} <\infty\}.
 \end{displaymath}
 The map $\|.\|_{p\textrm{-var},s,t}$ is a semi-norm on $C^{p\textrm{-var}}([s,t],\mathbb R^d)$.
 \\
 \\
 Moreover, $\|.\|_{p\textrm{-var},T} :=\|.\|_{p\textrm{-var},0,T}$.
 \\
 \\
 \textit{Remarks :}
 \begin{enumerate}[a.]
  \item For every $q,r\in [1,\infty[$ such that $q\geqslant r$,
  \begin{displaymath}
  \forall x\in C^{r\textrm{-var}}([s,t],\mathbb R^d)
  \textrm{, }
  \|x\|_{q\textrm{-var},s,t}
  \leqslant
  \|x\|_{r\textrm{-var},s,t}.
  \end{displaymath}
  In particular, any continuous function of bounded variation on $[s,t]$ belongs to $C^{q\textrm{-var}}([s,t],\mathbb R^d)$ for every $q\in [1,\infty[$.
  \item For every $(s,t)\in\Delta_T$ and $x\in C^{1\textrm{-var}}([s,t],\mathbb R)$,
  \begin{displaymath}
  \|x\|_{1\textrm{-var},s,t} =
  \int_{s}^{t}|Dx|,
  \end{displaymath}
  where $|Dx|$ is the variation measure of the differential measure $Dx$ associated with $x$.
 \end{enumerate}
 \item The vector space of Lipschitz continuous maps from $\mathbb R^e$ into $\mathcal M_{e,d}(\mathbb R)$ is denoted by $\textrm{Lip}(\mathbb R^e,\mathcal M_{e,d}(\mathbb R))$ and equipped with the Lipschitz semi-norm $\|.\|_{\textrm{Lip}}$ defined by
 \begin{displaymath}
 \|\varphi\|_{\textrm{Lip}} :=
 \sup\left\{
 \frac{\|\varphi(y) -\varphi(x)\|}{\|y - x\|}
 \textrm{ ; }
 x,y\in\mathbb R^e
 \textrm{ and }
 x\not= y\right\}
 \end{displaymath}
 for every $\varphi\in\textrm{Lip}(\mathbb R^e,\mathcal M_{e,d}(\mathbb R))$.
 \item For every $\lambda\in\mathbb R$,
 \begin{displaymath}
 \lfloor\lambda\rfloor :=
 \max\{n\in\mathbb Z : n <\lambda\}
 \end{displaymath}
 and $\{\lambda\} :=\lambda -\lfloor\lambda\rfloor$.
 \item Consider $\gamma\in [1,\infty[$. A continuous map $\varphi :\mathbb R^e\rightarrow\mathcal M_{d,e}(\mathbb R)$ is $\gamma$-Lipschitz in the sense of Stein if and only if,
 \begin{eqnarray*}
  \|\varphi\|_{\textrm{Lip}^{\gamma}} & := &
  \|D^{\lfloor\gamma\rfloor}\varphi\|_{\{\gamma\}\textrm{-H\"ol}}\vee
  \max\{\|D^k\varphi\|_{\infty}\textrm{ $;$ }k\in\llbracket 0,\lfloor\gamma\rfloor\rrbracket\}\\
  & < & \infty.
 \end{eqnarray*}
 Consider the vector space
 \begin{displaymath}
 \textrm{Lip}^{\gamma}(\mathbb R^e,\mathcal M_{e,d}(\mathbb R)) :=
 \{\varphi\in C^0(\mathbb R^e,\mathcal M_{e,d}(\mathbb R)) :\|\varphi\|_{\textrm{Lip}^{\gamma}} <\infty\}.
 \end{displaymath}
 The map $\|.\|_{\textrm{Lip}^{\gamma}}$ is a norm on $\textrm{Lip}^{\gamma}(\mathbb R^e,\mathcal M_{e,d}(\mathbb R))$.
 \\
 \\
 \textit{Remarks}:
 \begin{enumerate}[a.]
  \item If $\varphi\in\textrm{Lip}^{\gamma}(\mathbb R^e,\mathcal M_{e,d}(\mathbb R))$, then $\varphi\in\textrm{Lip}(\mathbb R^e,\mathcal M_{e,d}(\mathbb R))$.
  \item If $\varphi\in C^{\lfloor\gamma\rfloor + 1}(\mathbb R^e,\mathcal M_{e,d}(\mathbb R))$ is bounded with bounded derivatives, then $\varphi\in\textrm{Lip}^{\gamma}(\mathbb R^e,\mathcal M_{e,d}(\mathbb R))$.
 \end{enumerate}
\end{enumerate}
%


%
\section{Preliminaries}
This section deals with some preliminaries on sweeping processes and
the rough integral. The first subsection states some fundamental
results on unperturbed sweeping processes coming from Moreau
\cite{MOREAU76}, Valadier \cite{VALADIER90} and Monteiro Marques
\cite{MONTEIRO_MARQUES93}. A continuity result of  Castaing et
al.~\cite{CMR15}, which is the cornerstone of the proofs of Theorem
\ref{existence_Young} and Theorem \ref{existence_rough}, is also
stated. The second subsection deals with the integration along rough
paths. In this paper, definitions and propositions are stated as in
Friz and Hairer \cite{FH14}, in accordance with M.~Gubinelli's approach (see Gubinelli \cite{GUBINELLI04}).
%


%
\subsection{Sweeping processes}
The following theorem, due to Monteiro Marques
\cite{MONTEIRO_MARQUES84,MONTEIRO_MARQUES86,MONTEIRO_MARQUES93} 
using an estimation due to  Valadier (see
\cite{CASTAING83,VALADIER90}), 
states a sufficient condition of existence and uniqueness of the solution of the unperturbed sweeping process defined by Problem (\ref{sweeping_process}).
%


%
\begin{proposition}\label{Valadier_theorem}
Assume that $C$ is a convex compact valued multifunction, continuous
for the Hausdorff distance, and such that there exists $(x,r)\in\mathbb R^e\times ]0,\infty[$ satisfying
\begin{displaymath}
\overline B_e(x,r)\subset C(t)
\textrm{ $;$ }
\forall t\in [0,T].
\end{displaymath}
Then Problem (\ref{sweeping_process}) has a unique continuous solution of finite $1$-variation $y : [0,T]\rightarrow\mathbb R^e$ such that
\begin{displaymath}
\|y\|_{1\normalfont{\textrm{-var}},T}
\leqslant l(r,\|a - x\|),
\end{displaymath}
where $l :\mathbb R_{+}^{2}\rightarrow\mathbb R_+$ is the map defined by
\begin{displaymath}
l(s,S) :=
\left\{
\begin{array}{rcl}
\displaystyle{
\max\left\{0,\frac{S^2 - s^2}{2s}\right\}} & \textrm{if} & e > 1\\
\max\{0,S - s\} & \textrm{if} & e = 1
\end{array}
\right.
\textrm{$;$ }
\forall s,S\in\mathbb R_+.
\end{displaymath}
\end{proposition}
\noindent
This proposition is a consequence of the two following ones. These two propositions are also used in Section 5.
%


%
\begin{proposition}\label{characterization_sweeping_process}
Under Assumption \ref{assumption_C}, a map $y : [0,T]\rightarrow\mathbb R^e$ is a solution of Problem (\ref{sweeping_process}) if it satisfies the two following conditions:
\begin{enumerate}
 \item For every $t\in [0,T]$, $y(t)\in C(t)$.
 \item For every $(s,t)\in\Delta_T$ and $z\in\cap_{\tau\in [s,t]}C(\tau)$,
 \begin{displaymath}
 \langle z,y(t) - y(s)\rangle
 \geqslant\frac{1}{2}(\|y(t)\|^2 - \|y(s)\|^2).
 \end{displaymath}
\end{enumerate}
\end{proposition}
%


%
\begin{proposition}\label{approximation_scheme_SP}
Consider $n\in\mathbb N^*$, $(t_{0}^{n},\dots,t_{n}^{n})$ the dissection of $[0,T]$ of constant mesh $T/n$ and the step function $Y^n$ defined by
\begin{displaymath}
\left\{
\begin{array}{rcl}
 Y_{0}^{n} & := & a\\
 Y_{k + 1}^{n} & = &
 p_{C(t_{k + 1}^{n})}(Y_{k}^{n})\textrm{ $;$ }k\in\llbracket 0,n - 1\rrbracket\\
 Y^n(t) & := & Y_{k}^{n}\textrm{ $;$ } t\in [t_{k}^{n},t_{k + 1}^{n}[\textrm{$,$ }k\in\llbracket 0,n - 1\rrbracket.
\end{array}
\right.
\end{displaymath}
\begin{enumerate}
 \item Under the conditions of Proposition \ref{Valadier_theorem} on $C$, $\|Y^n\|_{1\normalfont{\textrm{-var}},T}\leqslant l(r,\|a - x\|)$.
 \item Under Assumption \ref{assumption_C}, for every $m\in n\mathbb N^*$ and $t\in [0,T]$, there exist $i\in\llbracket 1,n\rrbracket$ and $j\in\llbracket 1,m\rrbracket$ such that $t\in [t_{i - 1}^{n},t_{i}^{n}[$, $t\in [t_{j - 1}^{m},t_{j}^{m}[$ and
 \begin{displaymath}
 \|Y^n(t) - Y^m(t)\|^2
 \leqslant 2d_H(C(t_{i}^{n}),C(t_{j}^{m}))
 (\|Y^n\|_{1\normalfont{\textrm{-var}},t} +
 \|Y^m\|_{1\normalfont{\textrm{-var}},t}).
 \end{displaymath}
\end{enumerate}
\end{proposition}
\noindent
See Monteiro Marques \cite{MONTEIRO_MARQUES93}, Chapter 2 for the proofs of the three previous propositions.
\\
\\
Let $h$ be a continuous function from $[0,T]$ into $\mathbb R^e$ such that $h(0) = 0$. If it exists, a Skorokhod decomposition of $(C,a,h)$ is a couple $(v_h,w_h)$ such that:
\begin{equation}\label{Skorokhod_decomposition}
\left\{
  \begin{array}{rcl}
 v_h(t) & = & h(t) + w_h(t)\\
 -\displaystyle{\frac{dDw_h}{d|Dw_h|}}(t) & \in & N_{C_h(t)}(w_h(t))
 \textrm{ $|Dw_h|$-a.e. with }
 w_h(0) = a,
\end{array}
\right.
\end{equation}
where $v_h$ and $w_h$ are continuous, and $w_h$ has bounded
variation. 
Since $N_{C_h(t)}(x)=\emptyset$ when $x\not\in C_h(t)$,
the system \eqref{Skorokhod_decomposition} implies that, $|Dw_h|$-a.e.,
$w_h(t)\in C_h(t)$, that
is, $v_h(t)\in C(t)$.
Under Assumption \ref{assumption_C}, by Proposition \ref{Valadier_theorem} together with Castaing et al.~\cite[Lemma 2.2]{CMR15}, $(C,a,h)$ has a unique Skorokhod decomposition $(v_h,w_h)$.
%


%
\begin{theorem}\label{PRF_key_result}
Under Assumption \ref{assumption_C}, if $(h_n)_{n\in\mathbb N}$ is a sequence of continuous functions from $[0,T]$ into $\mathbb R^e$ which converges uniformly to $h\in C^0([0,T],\mathbb R^e)$, then
\begin{displaymath}
\sup_{n\in\mathbb N}
\|w_{h_n}\|_{1\normalfont{\textrm{-var}},T} <\infty
\end{displaymath}
and
\begin{displaymath}
(v_{h_n},w_{h_n})
\xrightarrow[n\rightarrow\infty]{\|.\|_{\infty,T}}
(v_h,w_h).
\end{displaymath}
\end{theorem}
\noindent
See Castaing et al.~\cite[Theorem 2.3]{CMR15}.
\\
\\
Under Assumption \ref{assumption_C}, note that there exist $R > 0$, $N\in\mathbb N^*$ and a dissection $(t_0,\dots,t_N)$ of $[0,T]$ such that
\begin{equation}\label{dissection_C}
\overline B_e(\gamma(t_k),R)\subset
C(u)
\end{equation}
for every $k\in\llbracket 0,N - 1\rrbracket$ and $u\in [t_k,t_{k + 1}]$.
%


%
\begin{proposition}\label{regularity_control_SD}
Under Assumption \ref{assumption_C}:
\begin{enumerate}
 \item The map $(v_.,w_.)$ is continuous from
 \begin{displaymath}
 C_{0}^{0}([0,T],\mathbb R^e)
 \textrm{ to }
 C^0([0,T],\mathbb R^e)\times
 C^{1\normalfont{\textrm{-var}}}([0,T],\mathbb R^e).
 \end{displaymath}
 \item Consider $(s,t)\in\Delta_T$ and $\rho\in ]0,R/2]$ where $R$ is defined in (\ref{dissection_C}). For every $h\in C_{0}^{0}([s,t],\mathbb R^e)$ such that $\|h\|_{0,s,t}\leqslant\rho$,
 \begin{displaymath}
 \|w_h\|_{1\normalfont{\textrm{-var}},s,t}
 \leqslant\mathfrak M(\rho)
 \end{displaymath}
 with
 \begin{displaymath}
 \mathfrak M(\rho) :=
 \frac{N}{2\rho}
 \sup_{u\in [0,T]}
 \sup_{x,y\in C(u)}\|x - y\|^2.
 \end{displaymath}
\end{enumerate}
\end{proposition}
%


%
\begin{proof}
Refer to Castaing et al.~\cite[Lemma 5.3]{CMR15} for a proof of the
first point.
\\
\\
Let us insert $s$ and $t$ in the dissection $(t_0,\dots,t_N)$ of $[0,T]$ and define $k(s),k(t)\in\llbracket 0,N + 2\rrbracket$ by
\begin{displaymath}
t_{k(s)} := s
\textrm{ and }
t_{k(t)} := t.
\end{displaymath}
Consider $k\in\llbracket k(s),k(t) - 1\rrbracket$ and $u\in [t_k,t_{k + 1}]$.
\\
\\
On the one hand,
\begin{displaymath}
\overline B_e(\gamma(t_k) - h(t_k),\rho)
\subset B_e(\gamma(t_k) - h(u),R)
\subset C_h(u).
\end{displaymath}
So,
\begin{displaymath}
\overline B_e(\gamma(t_k) - v_h(t_k),\rho)\subset
C(u) - h(t_k,u) - v_h(t_k).
\end{displaymath}
On the other hand,
\begin{displaymath}
v_h(t_k,u) = h(t_k,u) + w_{h,t_k}(u)
\end{displaymath}
with
\begin{displaymath}
w_{h,t_k}(u) :=
w_h(u) - w_h(t_k).
\end{displaymath}
Moreover,
\begin{displaymath}
-\frac{dDw_h}{d|Dw_h|}(u)\in N_{C(u) - h(u)}(w_h(u))
\textrm{ $|Dw_h|$-a.e.}
\end{displaymath}
and then,
\begin{displaymath}
\left\{
\begin{array}{rcl}
 \displaystyle{-\frac{dDw_{h,t_k}}{d|Dw_{h,t_k}|}(u)} & \in & N_{C(u) - h(t_k,u) - v_h(t_k)}(w_{h,t_k}(u))
 \textrm{ $|Dw_{h,t_k}|$-a.e.}\\
 w_{h,t_k}(t_k) & = & 0.
\end{array}\right. 
\end{displaymath}
So, by Proposition \ref{Valadier_theorem}:
\begin{eqnarray*}
 \|w_h\|_{1\textrm{-var},t_k,t_{k + 1}} & = &
 \|w_{h,t_k}\|_{1\textrm{-var},t_k,t_{k + 1}}\\
 & \leqslant &
 l(\rho,\|\gamma(t_k) - v_h(t_k)\|).
\end{eqnarray*}
Therefore,
\begin{eqnarray*}
 \|w_h\|_{1\textrm{-var},s,t} & = &
 \sum_{k = k(s)}^{k(t) - 1}
 \|w_h\|_{1\textrm{-var},t_k,t_{k + 1}}\\
 & \leqslant &
 N\sup_{u\in [s,t]}
 l(\rho,\|\gamma(u) - v_h(u)\|)\\
 & \leqslant &
 \mathfrak M(\rho).
\end{eqnarray*}
\end{proof}
%


%
\subsection{Young's integral, rough integral}
The first part of the subsection deals with the definition and some basic properties of Young's integral which allow to integrate a map $y\in C^{r\normalfont{\textrm{-var}}}([0,T],\mathcal M_{e,d}(\mathbb R))$ with respect to $z\in C^{q\normalfont{\textrm{-var}}}([0,T],\mathbb R^d)$ when $q,r\in [1,\infty[$ and $1/q + 1/r > 1$. The second part of the subsection deals with the rough integral which extends Young's integral when the condition $1/q + 1/r > 1$ is not satisfied anymore. The signal $z$ has to be enhanced as a rough path.
%


%
\begin{definition}\label{control_function}
A map $\omega :\Delta_T\rightarrow\mathbb R_+$ is a control function if and only if,
\begin{enumerate}
 \item $\omega$ is continous.
 \item $\omega(s,s) = 0$ for every $s\in [0,T]$.
 \item $\omega$ is super-additive:
 \begin{displaymath}
 \omega(s,u) +
 \omega(u,t)\leqslant
 \omega(s,t)
 \end{displaymath}
 for every $s,t,u\in [0,T]$ such that $s\leqslant u\leqslant t$.
 \end{enumerate}
\end{definition}
\noindent
\textbf{Example.} Let $p\geqslant 1$.
For every $z\in C^{p\textrm{-var}}([0,T],\mathbb R^d)$, the map
\begin{displaymath}
\omega_{p,z} : (s,t)\in\Delta_T
\longmapsto
\omega_{p,z}(s,t) :=\|z\|_{p\textrm{-var},s,t}^{p}
\end{displaymath}
is a control function.
%


%
\begin{proposition}\label{uniform_convergence_p_variation}
Let $p\geqslant 1$.
Consider $x\in C^0([0,T],\mathbb R^d)$ and a sequence $(x_n)_{n\in\mathbb N}$ of elements of $C^{p\normalfont{\textrm{-var}}}([0,T],\mathbb R^d)$ such that
\begin{displaymath}
\lim_{n\rightarrow\infty}
\|x_n - x\|_{\infty,T} = 0
\textrm{ and }
\sup_{n\in\mathbb N}
\|x_n\|_{p\normalfont{\textrm{-var}},T}
< \infty.
\end{displaymath}
Then $x\in C^{p\normalfont{\textrm{-var}}}([0,T],\mathbb R^d)$ and
\begin{displaymath}
\lim_{n\rightarrow\infty}
\|x_n - x\|_{(p +\varepsilon)\normalfont{\textrm{-var}},T} = 0
\textrm{ $;$ }
\forall\varepsilon > 0.
\end{displaymath}
\end{proposition}
\noindent
See Friz and Victoir \cite[Lemma 5.12 and Lemma 5.27]{FV10} for a proof.
%


%
\begin{proposition}\label{Young_integral}
(Young's integral) Consider $q,r\in [1,\infty[$ such that $1/q + 1/r > 1$, and two maps $y\in C^{r\normalfont{\textrm{-var}}}([0,T],\mathcal M_{e,d}(\mathbb R))$ and $z\in C^{q\normalfont{\textrm{-var}}}([0,T],\mathbb R^d)$. For every $n\in\mathbb N^*$ and $(t_{k}^{n})_{k\in\llbracket 1,n\rrbracket}\in\mathfrak D_{[0,T]}$, the limit
\begin{displaymath}
\lim_{n\rightarrow\infty}
\sum_{k = 1}^{n - 1}
y(t_{k}^{n})
z(t_{k}^{n},t_{k + 1}^{n})
\end{displaymath}
exists and does not depend on the dissection $(t_{k}^{n})_{k\in\llbracket 1,n\rrbracket}$. That limit is denoted by
\begin{displaymath}
\int_{0}^{T}y(s)dz(s)
\end{displaymath}
and called Young's integral of $y$ with respect to $z$ on $[0,T]$. Moreover, there exists a constant $c(q,r) > 0$, depending only on $q$ and $r$, such that for every $(s,t)\in\Delta_T$,
\begin{displaymath}
\left\|
\int_{0}^{.}y(s)dz(s)\right\|_{r\normalfont{\textrm{-var}},s,t}
\leqslant
c(q,r)\|z\|_{q\normalfont{\textrm{-var},s,t}}(
\|y\|_{r\normalfont{\textrm{-var},s,t}} +\|y\|_{\infty,s,t}).
\end{displaymath}
\end{proposition}
\noindent
See Lyons \cite[Theorem 1.16]{LYONS04}, Lejay \cite[Theorem
1]{LEJAY10} or Friz and Victoir \cite[Theorem 6.8]{FV10} for a
proof. 
%


%
\begin{proposition}\label{continuity_Young_integral}
Consider $q,r\in [1,\infty[$ such that $1/q + 1/r > 1$, two maps $y\in C^{r\normalfont{\textrm{-var}}}([0,T],\mathcal M_{e,d}(\mathbb R))$ and $z\in C^{q\normalfont{\textrm{-var}}}([0,T],\mathbb R^d)$, and a sequence $(y_n)_{n\in\mathbb N}$ of elements of $C^{r\normalfont{\textrm{-var}}}([0,T],\mathcal M_{e,d}(\mathbb R))$ such that:
\begin{displaymath}
\lim_{n\rightarrow\infty}
\|y_n - y\|_{\infty,T}=0
\textrm{ and }
\sup_{n\in\mathbb N}
\|y_n\|_{r\normalfont{\textrm{-var}},T} <\infty.
\end{displaymath}
Then,
\begin{displaymath}
\lim_{n\rightarrow\infty}
\left\|\int_{0}^{.}y_n(s)dz(s) -\int_{0}^{.}y(s)dz(s)\right\|_{\infty,T} = 0.
\end{displaymath}
\end{proposition}
\noindent
See Friz and Victoir \cite[Proposition 6.12]{FV10} for a proof.
\\
\\
Consider $p\in [2,3[$ and let us define the rough integral for continuous functions of finite $p$-variation.
\\
\\
\textbf{Remark.} In the sequel, the reader has to keep in mind that:
\begin{enumerate}
 \item $\mathcal M_{e,d}(\mathbb R)\cong \mathbb R^e\otimes\mathbb R^d$.
 \item $\mathcal M_{d,1}(\mathbb R)\cong\mathcal M_{1,d}(\mathbb R)\cong\mathbb R^d$.
 \item $\mathcal L(\mathbb R^d,\mathcal M_{e,d}(\mathbb R))\cong\mathcal L(\mathbb R^d,\mathcal L(\mathbb R^d,\mathbb R^e))
\cong\mathcal L(\mathbb R^d\otimes\mathbb R^d,\mathbb R^e)$.
\end{enumerate}
%


%
\begin{definition}\label{signature}
Consider $z\in C^{1\normalfont{\textrm{-var}}}([0,T],\mathbb R^d)$. The step-2 signature of $z$ is the map $S_2(z) :\Delta_T\rightarrow\mathbb R^d\times\mathcal M_d(\mathbb R)$ defined by
\begin{displaymath}
S_2(z)(s,t) :=
\left(z(s,t),\int_{s < u < v < t}dz(v)\otimes dz(u)\right)
\end{displaymath}
for every $(s,t)\in\Delta_T$.
\end{definition}
\noindent
\textbf{Notation.} $\mathfrak S_{T}(\mathbb R^d) :=
\{S_2(z)(0,.)\textrm{ $;$ }z\in C^{1\normalfont{\textrm{-var}}}([0,T],\mathbb R^d)\}$.
%


%
\begin{definition}\label{geometric_rough_paths}
The geometric $p$-rough paths metric space $G\Omega_{p,T}(\mathbb R^d)$ is the closure of $\mathfrak S_{T}(\mathbb R^d)$ in $C^{p\normalfont{\textrm{-var}}}([0,T],\mathbb R^d)\times C^{p/2\normalfont{\textrm{-var}}}([0,T],\mathcal M_d(\mathbb R))$.
\end{definition}
%


%
\begin{definition}\label{controlled_rough_paths}
For $z\in C^{p\normalfont{\textrm{-var}}}([0,T],\mathbb R^d)$, a map $y\in C^{p\normalfont{\textrm{-var}}}([0,T],\mathcal M_{e,d}(\mathbb R))$ is controlled by $z$ if and only if there exists $y'\in C^{p\normalfont{\textrm{-var}}}([0,T],\mathcal L(\mathbb R^d,\mathcal M_{e,d}(\mathbb R)))$ such that
\begin{displaymath}
y(s,t) = y'(s)z(s,t) + R_y(s,t)
\textrm{ $;$ }
\forall (s,t)\in\Delta_T
\end{displaymath}
with $\|R_y\|_{p/2\normalfont{\textrm{-var}},T} <\infty$.  For fixed $z$, the pairs $(y,y')$ as above define a vector space denoted by $\mathfrak D_{z}^{p/2}([0,T],\mathcal M_{e,d}(\mathbb R))$ and equipped with the semi-norm $\|.\|_{z,p/2,T}$ such that
\begin{displaymath}
\|(y,y')\|_{z,p/2,T} :=
\|y'\|_{p\normalfont{\textrm{-var}},T} +
\|R_y\|_{p/2\normalfont{\textrm{-var}},T}
\end{displaymath}
for every $(y,y')\in\mathfrak D_{z}^{p/2}([0,T],\mathcal
M_{e,d}(\mathbb R))$.
\end{definition}
%


%
\begin{theorem}\label{rough_integral}
(Rough integral) Consider $\mathbf z := (z,\mathbb Z)\in G\Omega_{p,T}(\mathbb R^d)$ and $(y,y')\in\mathfrak D_{z}^{p/2}([0,T],\mathcal M_{e,d}(\mathbb R))$. For every $n\in\mathbb N^*$ and $(t_{k}^{n})_{k\in\llbracket 1,n\rrbracket}\in\mathfrak D_{[0,T]}$, the limit
\begin{displaymath}
\lim_{n\rightarrow\infty}
\sum_{k = 1}^{n - 1}
(y(t_{k}^{n})z(t_{k}^{n},t_{k + 1}^{n}) +
y'(t_{k}^{n})\mathbb Z(t_{k}^{n},t_{k + 1}^{n}))
\end{displaymath}
exists and does not depend on the dissection $(t_{k}^{n})_{k\in\llbracket 1,n\rrbracket}$. That limit is denoted by
\begin{displaymath}
\int_{0}^{T}y(s)d\mathbf z(s)
\end{displaymath}
and called rough integral of $y$ with respect to $\mathbf z$ on $[0,T]$. Moreover,
\begin{enumerate}
 \item There exists a constant $c(p) > 0$, depending only on $p$, such that for every $(s,t)\in\Delta_T$,
 \begin{eqnarray*}
  \left\|
  \int_{s}^{t}
  y(u)d\mathbf z(u)
  - y(s)z(s,t) - y'(s)\mathbb Z(s,t)
  \right\|
  & \leqslant &
  c(p)
  (\|z\|_{p\normalfont{\textrm{-var},s,t}}\|R_y\|_{p/2\normalfont{\textrm{-var}},s,t}\\
  & &
  +\|\mathbb Z\|_{p/2\normalfont{\textrm{-var},s,t}}\|y'\|_{p\normalfont{\textrm{-var}},s,t}).
 \end{eqnarray*}
 \item The map
 \begin{displaymath}
 (y,y')\longmapsto
 \left(\int_{0}^{.}y(s)d\mathbf z(s),y\right)
 \end{displaymath}
 is continuous from $\mathfrak D_{z}^{p/2}([0,T],\mathcal M_{e,d}(\mathbb R))$ into $\mathfrak D_{z}^{p/2}([0,T],\mathbb R^e)$.
\end{enumerate}
\end{theorem}
\noindent
See Friz and Shekhar \cite[Theorem 34]{FS14} for a proof with the
$p$-variation topology, and see Gubinelli \cite[Theorem
1]{GUBINELLI04} or Friz and Hairer \cite[Theorem 4.10]{FH14} for a
proof with the $1/p$-H\"older topology. 
%


%
\begin{proposition}\label{continuity_rough_integral}
Consider $\mathbf z := (z,\mathbb Z)\in G\Omega_{p,T}(\mathbb R^d)$, a continuous map
\begin{displaymath}
(y,y') : [0,T]\longrightarrow
\mathcal M_{e,d}(\mathbb R)\times
\mathcal L(\mathbb R^d,\mathcal M_{e,d}(\mathbb R)),
\end{displaymath}
and a sequence $(y_n,y_n')_{n\in\mathbb N}$ of elements of $\mathfrak D_{z}^{p/2}([0,T],\mathcal M_{e,d}(\mathbb R))$ such that
\begin{displaymath}
(y_n',R_{y_n})\xrightarrow[n\rightarrow\infty]{d_{\infty,T}} (y',R_y)
\textrm{ and }
\sup_{n\in\mathbb N}
\|(y_n,y_n')\|_{z,p/2,T} <\infty.
\end{displaymath}
Then, $(y,y')\in\mathfrak D_{z}^{p/2}([0,T],\mathcal M_{e,d}(\mathbb R))$ and
\begin{displaymath}
\lim_{n\rightarrow\infty}
\left\|\int_{0}^{.}y_n(s)d\mathbf z(s) -\int_{0}^{.}y(s)d\mathbf z(s)\right\|_{\infty,T} = 0.
\end{displaymath}
\end{proposition}
%


%
\begin{proof}
On the one hand, since
\begin{displaymath}
(y_n',R_{y_n})\xrightarrow[n\rightarrow\infty]{d_{\infty,T}} (y',R_y),
\end{displaymath}
the function $y$ is the uniform limit of the sequence $(y_n)_{n\in\mathbb N}$. Moreover, since
\begin{displaymath}
\sup_{n\in\mathbb N}
\|(y_n,y_n')\|_{z,p/2,T} <\infty,
\end{displaymath}
by Proposition \ref{uniform_convergence_p_variation},
\begin{displaymath}
y'\in C^{p\textrm{-var}}([0,T],\mathcal L(\mathbb R^d,\mathcal M_{e,d}(\mathbb R)))
\textrm{ and }
R_y\in C^{p/2\textrm{-var}}([0,T],\mathcal M_{e,d}(\mathbb R)).
\end{displaymath}
So, $(y,y')\in\mathfrak D_{z}^{p/2}([0,T],\mathcal M_{e,d}(\mathbb R))$.
\\
\\
On the other hand, also by Proposition \ref{uniform_convergence_p_variation}, for any $\varepsilon > 0$ such that $p +\varepsilon\in [2,3[$,
\begin{eqnarray*}
 \lim_{n\rightarrow\infty}
 \|(y_n,y_n') - (y,y')\|_{z,(p +\varepsilon)/2,T} & = &
 \lim_{n\rightarrow\infty}
 \|y_n' - y'\|_{(p +\varepsilon)\textrm{-var},T}\\
 & &
+ \lim_{n\rightarrow\infty}
 \|R_{y_n} - R_y\|_{(p +\varepsilon)/2\textrm{-var},T}\\
 & = & 0.
\end{eqnarray*}
So, by continuity of the rough integral (see Theorem \ref{rough_integral}),
\begin{displaymath}
\lim_{n\rightarrow\infty}
\left\|\left(\int_{0}^{.}y_n(s)d\mathbf z(s),y_n\right) -
\left(\int_{0}^{.}y(s)d\mathbf z(s),y\right)\right\|_{z,(p +\varepsilon)/2,T} = 0.
\end{displaymath}
Therefore, in particular:
\begin{displaymath}
\lim_{n\rightarrow\infty}
\left\|\int_{0}^{.}y_n(s)d\mathbf z(s) -\int_{0}^{.}y(s)d\mathbf z(s)\right\|_{\infty,T} = 0.
\end{displaymath}
\end{proof}
%


%
\begin{proposition}\label{composition_rough_regular}
Consider $\mathbf z := (z,\mathbb Z)\in G\Omega_{p,T}(\mathbb R^d)$, $(x,x')\in\mathfrak D_{z}^{p/2}([0,T],\mathbb R^e)$ and $\varphi\in\normalfont{\textrm{Lip}}^{\gamma - 1}(\mathbb R^e,\mathbb R^d)$. The couple of maps $(\varphi(x),\varphi(x)')$, defined by
\begin{displaymath}
\varphi(x)(t) :=\varphi(x(t))
\textrm{ and }
\varphi(x)'(t) := D\varphi(x(t))x'(t)
\end{displaymath}
for every $t\in [0,T]$, belongs to $\mathfrak D_{z}^{p/2}([0,T],\mathcal M_{e,d}(\mathbb R))$.
\end{proposition}
\noindent
\textbf{Remark.} By Theorem \ref{rough_integral} and Proposition \ref{composition_rough_regular} together,
\begin{displaymath}
\int_{0}^{.}\varphi(x(u))d\mathbf z(u)
\end{displaymath}
is defined. For every $(s,t)\in\Delta_T$, consider
\begin{displaymath}
\mathfrak I_{\varphi,\mathbf z,x}(s,t) :=
\left\|
\int_{s}^{t}
\varphi(x(u))d\mathbf z(u)
-\varphi(x(s))z(s,t) - D\varphi(x(s))x'(s)\mathbb Z(s,t)
\right\|.
\end{displaymath}
For every $(s,t)\in\Delta_T$, since
\begin{eqnarray*}
 \|\varphi(x)\|_{p\textrm{-var},s,t}
 & \leqslant &
 \|\varphi\|_{\textrm{Lip}^{\gamma - 1}}\|x\|_{p\textrm{-var},s,t},\\
 \|\varphi(x)'\|_{p\textrm{-var},s,t}
 & \leqslant &
 \|\varphi\|_{\textrm{Lip}^{\gamma - 1}}(\|x'\|_{p\textrm{-var},s,t} +\|x'\|_{\infty,s,t}\|x\|_{p\textrm{-var},s,t})
 \textrm{ and}\\
 \|R_{\varphi(x)}\|_{p/2\textrm{-var},s,t}
 & \leqslant &
 \|\varphi\|_{\textrm{Lip}^{\gamma - 1}}(\|x\|_{p\textrm{-var},s,t}^{2} +\|R_x\|_{p/2\textrm{-var},s,t}),
\end{eqnarray*}
by Theorem \ref{rough_integral},
\begin{eqnarray*}
 \mathfrak I_{\varphi,\mathbf z,x}(s,t)
 & \leqslant &
 c(p)\|\varphi\|_{\textrm{Lip}^{\gamma - 1}}(
 \|x'\|_{p\textrm{-var},s,t} +\|x'\|_{\infty,s,t}\|x\|_{p\textrm{-var},s,t}\\
 & &
 +\|x\|_{p\textrm{-var},s,t}^{2} +\|R_x\|_{p/2\textrm{-var},s,t})\omega_{p,\mathbf z}(s,t)^{1/p},
\end{eqnarray*}
where $\omega_{p,\mathbf z} :\Delta_T\rightarrow\mathbb R_+$ is the control function defined by
\begin{displaymath}
\omega_{p,\mathbf z}(u,v) :=
2^{p - 1}(\|z\|_{p\textrm{-var},u,v}^{p} +
\|\mathbb Z\|_{p/2\textrm{-var},u,v}^{p})
\textrm{ $;$ }
\forall (u,v)\in\Delta_T.
\end{displaymath}
%


%
\begin{proposition}\label{continuity_rough_integral_composition}
Consider $\mathbf z := (z,\mathbb Z)\in G\Omega_{p,T}(\mathbb R^d)$, $\varphi\in\normalfont{\textrm{Lip}}^{\gamma - 1}(\mathbb R^e,\mathbb R^d)$, a continuous map
\begin{displaymath}
(x,x') : [0,T]\longrightarrow
\mathbb R^e\times
\mathcal L(\mathbb R^d,\mathbb R^e),
\end{displaymath}
and a sequence $(x_n,x_n')_{n\in\mathbb N}$ of elements of $\mathfrak D_{z}^{p/2}([0,T],\mathbb R^e)$ such that
\begin{displaymath}
(x_n',R_{x_n})\xrightarrow[n\rightarrow\infty]{d_{\infty,T}} (x',R_x)
\textrm{ and }
\sup_{n\in\mathbb N}
\|(x_n,x_n')\|_{z,p/2,T} <\infty.
\end{displaymath}
Then, $(\varphi(x),\varphi(x)')\in\mathfrak D_{z}^{p/2}([0,T],\mathcal M_{e,d}(\mathbb R))$ and
\begin{displaymath}
\lim_{n\rightarrow\infty}
\left\|\int_{0}^{.}\varphi(x_n(s))d\mathbf z(s) -\int_{0}^{.}\varphi(x(s))d\mathbf z(s)\right\|_{\infty,T} = 0.
\end{displaymath}
\end{proposition}
%


%
\begin{proof}
Since
\begin{displaymath}
(x_n',R_{x_n})\xrightarrow[n\rightarrow\infty]{d_{\infty,T}} (x',R_x)
\textrm{ and }
\sup_{n\in\mathbb N}
\|(x_n,x_n')\|_{z,p/2,T} <\infty,
\end{displaymath}
by Friz and Hairer \cite[Theorem 7.5]{FH14} together with Proposition \ref{uniform_convergence_p_variation},
\begin{displaymath}
(\varphi(x_n)',R_{\varphi(x_n)})\xrightarrow[n\rightarrow\infty]{d_{\infty,T}} (\varphi(x)',R_{\varphi(x)})
\textrm{ and }
\sup_{n\in\mathbb N}
\|(\varphi(x_n),\varphi(x_n)')\|_{z,p/2,T} <\infty.
\end{displaymath}
So, by Proposition \ref{continuity_rough_integral},
\begin{displaymath}
\lim_{n\rightarrow\infty}
\left\|\int_{0}^{.}\varphi(x_n(s))d\mathbf z(s) -\int_{0}^{.}\varphi(x(s))d\mathbf z(s)\right\|_{\infty,T} = 0.
\end{displaymath}
\end{proof}
%


%
\section{Existence of solutions}
The existence of a solution to Problem (\ref{rough_sweeping_process}) is established in Theorem \ref{existence_Young} when $p\in [1,2[$, and in Theorem \ref{existence_rough} when $p\in [2,3[$.
%


%
\begin{theorem}\label{existence_Young}
Under Assumption \ref{assumption_C}, if $p\in [1,2[$,
Problem (\ref{rough_sweeping_process}) has at least one solution which belongs to $C^{p\normalfont{\textrm{-var}}}([0,T],\mathbb R^e)$.
\end{theorem}
%


%
\begin{proof}
Consider the discrete scheme
\begin{equation}\label{Skorokhod_problem_discretized_1}
\left\{
\begin{array}{rcl}
 X_n(t) & = &
 H_n(t) + Y_n(t)\\
 H_n(t) & = &
 \displaystyle{
 \int_{0}^{t}f(X_{n - 1}(s))dZ(s)}\\
 -\displaystyle{\frac{dDY_n}{d|DY_n|}}(t) & \in & N_{C_{H_n}(t)}(Y_n(t))
 \textrm{ $|DY_n|$-a.e. with }Y_n(0) = a
\end{array}
\right.
\end{equation}
for Problem (\ref{rough_sweeping_process}), initialized by
\begin{equation}\label{Skorokhod_problem_discretized_1_initial}
\left\{
\begin{array}{rcl}
 -\displaystyle{\frac{dDX_0}{d|DX_0|}}(t) & \in & N_{C(t)}(X_0(t))
 \textrm{ $|DX_0|$-a.e.}\\
 X_0(0) & = & a.
\end{array}
\right.
\end{equation}
Since the map $\|Z\|_{p\textrm{-var},0,.}$ is continuous from $[0,T]$ into $\mathbb R_+$, and since $\|Z\|_{p\textrm{-var},0,0} = 0$, there exists $\tau_0\in [0,T]$ such that
\begin{displaymath}
\|Z\|_{p\textrm{-var},\tau_0}\leqslant\mu :=
\frac{m}{c(p,p)\|f\|_{\textrm{Lip}^{\gamma}}(m + M + 1)},
\end{displaymath}
where $m := R/2$ and $M :=\mathfrak M(R/2)$ (see Proposition \ref{regularity_control_SD}.(2)). Let us show that for every $n\in\mathbb N$,
\begin{equation}\label{local_control_1}
\left\{
\begin{array}{rcl}
 \|X_n\|_{p\textrm{-var},\tau_0} & \leqslant & m + M\\
 \|H_n\|_{p\textrm{-var},\tau_0} & \leqslant & m\\
 \|Y_n\|_{1\textrm{-var},\tau_0} & \leqslant & M.
\end{array}
\right.
\end{equation}
By (\ref{Skorokhod_problem_discretized_1_initial}) together with Proposition \ref{regularity_control_SD},
\begin{displaymath}
\|X_0\|_{p\textrm{-var},\tau_0}
\leqslant M.
\end{displaymath}
Assume that Condition (\ref{local_control_1}) is satisfied for $n\in\mathbb N$ arbitrarily chosen. By Proposition \ref{Young_integral}, and since $\|Z\|_{p\textrm{-var},0,.}$ is an increasing map,
\begin{eqnarray*}
 \|H_{n + 1}\|_{p\textrm{-var},\tau_0}
 & \leqslant &
 c(p,p)\|Z\|_{p\textrm{-var},\tau_0}(
 \|Df\|_{\infty}\|X_n\|_{p\textrm{-var},\tau_0} +\|f\circ X_n\|_{\infty,\tau_0})\\
 & \leqslant &
 \mu c(p,p)\|f\|_{\textrm{Lip}^{\gamma}}(m + M + 1)\\
 & \leqslant & m.
\end{eqnarray*}
Since $Y_{n + 1} = w_{H_{n + 1}}$, by Proposition \ref{regularity_control_SD},
\begin{displaymath}
\|Y_{n + 1}\|_{1\textrm{-var},\tau_0}
\leqslant M.
\end{displaymath}
Therefore,
\begin{displaymath}
\|X_{n + 1}\|_{p\textrm{-var},\tau_0}
\leqslant
\|H_{n + 1}\|_{p\textrm{-var},\tau_0} +\|Y_{n + 1}\|_{p\textrm{-var},\tau_0}
\leqslant m + M.
\end{displaymath}
By induction, (\ref{local_control_1}) is satisfied for every $n\in\mathbb N$.
\\
\\
For every $t\in [0,T]$, the map $\|Z\|_{p\textrm{-var},t,.}$ is continuous from $[t,T]$ into $\mathbb R_+$ and $\|Z\|_{p\textrm{-var},t,t} = 0$. Moreover, the constant $\mu$ depends only on $p$, $m$, $M$ and $\|f\|_{\textrm{Lip}^{\gamma}}$. So, since $[0,T]$ is compact, there exist $N\in\mathbb N^*$ and $(\tau_k)_{k\in\llbracket 0,N\rrbracket}\in\mathfrak D_{[\tau_0,T]}$ such that
\begin{displaymath}
\|Z\|_{p\textrm{-var},\tau_k,\tau_{k + 1}}\leqslant\mu
\textrm{ $;$ }
\forall k\in\llbracket 0,N - 1\rrbracket.
\end{displaymath}
Since for every $n\in\mathbb N^*$ the maps
\begin{eqnarray*}
 (s,t)\in\Delta_T & \longmapsto &
 \|X_n\|_{p\textrm{-var},s,t}^{p},\\
 (s,t)\in\Delta_T & \longmapsto &
 \|H_n\|_{p\textrm{-var},s,t}^{p}
 \textrm{ and}\\
 (s,t)\in\Delta_T & \longmapsto &
 \|Y_n\|_{1\textrm{-var},s,t}
\end{eqnarray*}
are control functions, recursively, the sequence $(H_n,X_n,Y_n)_{n\in\mathbb N^*}$ is bounded in
\begin{displaymath}
\mathfrak C_{T}^{p,1} :=
C^{p\textrm{-var}}([0,T],\mathbb R^e)\times
C^{p\textrm{-var}}([0,T],\mathbb R^e)\times
C^{1\textrm{-var}}([0,T],\mathbb R^e).
\end{displaymath}
By Proposition \ref{Young_integral}, for every $n\in\mathbb N^*$ and $(s,t)\in\Delta_T$,
\begin{displaymath}
\|H_n(t) - H_n(s)\|\leqslant
c(p,p)\left(
\|Df\|_{\infty}\sup_{n\in\mathbb N}\|X_n\|_{p\textrm{-var},T} +
\|f\|_{\infty}\right)\|Z\|_{p\textrm{-var},s,t}.
\end{displaymath}
Since $(s,t)\in\Delta_T\mapsto\|Z\|_{p\textrm{-var},s,t}$ is a continuous map such that $\|Z\|_{p\textrm{-var},t,t} = 0$ for every $t\in [0,T]$, $(H_n)_{n\in\mathbb N^*}$ is equicontinuous. Therefore, by Arzel\`a-Ascoli's theorem together with Proposition \ref{uniform_convergence_p_variation}, there exists an extraction $\varphi :\mathbb N^*\rightarrow\mathbb N^*$ such that $(H_{\varphi(n)})_{n\in\mathbb N^*}$ converges uniformly to an element $H$ of $C^{p\textrm{-var}}([0,T],\mathbb R^e)$.
\\
\\
Since $(H_{\varphi(n)})_{n\in\mathbb N^*}$ converges uniformly to $H$, by Theorem \ref{PRF_key_result}, $(X_{\varphi(n)},Y_{\varphi(n)})_{n\in\mathbb N^*}$ converges uniformly to $(X,Y) := (v_H,w_H)$. So, for every $t\in [0,T]$,
\begin{displaymath}
\left\{
\begin{array}{rcl}
 X(t) & = & H(t) + Y(t)\\
 -\displaystyle{\frac{
 dDY}{d|DY|}}(t) & \in & N_{C_H(t)}(Y(t))
 \textrm{ $|DY|$-a.e. with }Y(0) = a,
\end{array}
\right.
\end{displaymath}
and by Proposition \ref{uniform_convergence_p_variation},
\begin{displaymath}
X\in C^{p\textrm{-var}}([0,T],\mathbb R^e)
\textrm{ and }
Y\in C^{1\textrm{-var}}([0,T],\mathbb R^e).
\end{displaymath}
Moreover, since $(X_{\varphi(n)})_{n\in\mathbb N^*}$ converges uniformly to $X$, by Proposition \ref{continuity_Young_integral},
\begin{displaymath}
\lim_{n\rightarrow\infty}
\left\|H_{\varphi(n)} -
\int_{0}^{.}f(X(s))dZ(s)
\right\|_{\infty,T} = 0.
\end{displaymath}
Therefore, since $(H_{\varphi(n)})_{n\in\mathbb N^*}$ converges also to $H$ in $C^0([0,T],\mathbb R^e)$,
\begin{displaymath}
H(t) =
\int_{0}^{t}f(X(s))dZ(s)
\textrm{ $;$ }
\forall t\in [0,T].
\end{displaymath}
\end{proof}
\noindent
In the sequel, assume that there exists $\mathbb Z : [0,T]\rightarrow\mathcal M_d(\mathbb R)$ such that $\mathbf Z := (Z,\mathbb Z)\in G\Omega_{p,T}(\mathbb R^d)$.
%


%
\begin{theorem}\label{existence_rough}
Under Assumption \ref{assumption_C}, if $p\in [2,3[$,
Problem (\ref{rough_sweeping_process}) has at least one solution which belongs to $C^{p\normalfont{\textrm{-var}}}([0,T],\mathbb R^e)$.
\end{theorem}
%


%
\begin{proof}
Consider the discrete scheme
\begin{equation}\label{Skorokhod_problem_discretized_2}
\left\{
\begin{array}{rcl}
 X_n(t) & = &
 H_n(t) + Y_n(t)\\
 H_n(t) & = &
 \displaystyle{
 \int_{0}^{t}f(X_{n - 1}(s))d\mathbf Z(s)}\\
 -\displaystyle{\frac{dDY_n}{d|DY_n|}}(t) & \in & N_{C_{H_n}(t)}(Y_n(t))
 \textrm{ $|DY_n|$-a.e. with }Y_n(0) = a
\end{array}
\right.
\end{equation}
for Problem (\ref{rough_sweeping_process}), initialized by
\begin{equation}\label{Skorokhod_problem_discretized_2_initial}
\left\{
\begin{array}{rcl}
 -\displaystyle{\frac{dDX_0}{d|DX_0|}(t)} & \in & N_{C(t)}(X_0(t))
 \textrm{ $|DX_0|$-a.e.}\\
 X_0(0) & = & a.
\end{array}
\right.
\end{equation}
Since the map $\omega_{p,\mathbf Z}(0,.)$ is continuous from $[0,T]$ into $\mathbb R_+$, and since $\omega_{p,\mathbf Z}(0,0) = 0$, there exists $\tau_0\in [0,T]$ such that
\begin{multline*}
 \omega_{p,\mathbf Z}(0,\tau_0)
  \leqslant 
 \frac{m_C}{c(p,1)^p\|f\|_{\textrm{Lip}^{\gamma}}^{p}(M_C + 1)^p}
\,\wedge\, \frac{m_C}{(c_{2}\vee c_6)^{p}(1 +\mu_C + M_R +\mu_{C}^{2})^p}\\
\wedge\, \frac{1}{1 +\mu_{C}^{p} + M_{R}^{p} +\mu_{C}^{2p}},
\end{multline*}
where $m_C := R/2$, $M_C :=\mathfrak M(R/2)$, $\mu_C := m_C + M_C$,
\begin{displaymath}
M_R := (c_1M_C4^{1/p})\vee(c_5(\mu_{C}^{p} + 1)^{1/p})
\end{displaymath}
and the positive constants $c_1$, $c_2$, $c_5$ and $c_6$, depending only on $p$ and $\|f\|_{\textrm{Lip}^{\gamma}}$, are defined in the sequel.
\\
\\
First of all, let us control the solution of the discrete scheme for $n\in\{0,1\}$:
\begin{itemize}
 \item ($n = 0$) By (\ref{Skorokhod_problem_discretized_2_initial}) together with Proposition \ref{regularity_control_SD}:
 \begin{displaymath}
 \|X_0\|_{1\textrm{-var},\tau_0}
 \leqslant M_C.
 \end{displaymath}
 \item ($n = 1$) Since $X_0\in C^{1\textrm{-var}}([0,T],\mathbb R^e)$, by Proposition \ref{Young_integral}:
 \begin{eqnarray*}
  \|H_1\|_{p\textrm{-var},\tau_0}
  & \leqslant &
  c(p,1)\omega_{p,\mathbf Z}(0,\tau_0)^{1/p}\|f\|_{\textrm{Lip}^{\gamma}}(M_C + 1)\\
  & \leqslant & m_C.
 \end{eqnarray*}
 Since $Y_1 = w_{H_1}$, by Proposition \ref{regularity_control_SD}:
 \begin{displaymath}
 \|Y_1\|_{1\textrm{-var},\tau_0}
 \leqslant M_C.
 \end{displaymath}
 Therefore,
 \begin{displaymath}
 \|X_1\|_{p\textrm{-var},\tau_0}
 \leqslant
 \|H_1\|_{p\textrm{-var},\tau_0} +
 \|Y_1\|_{p\textrm{-var},\tau_0}
 \leqslant
 \mu_C.
 \end{displaymath}
\end{itemize}
Let us show that for every $n\in\mathbb N\backslash\{0,1\}$,
\begin{equation}\label{local_control_2_1}
(X_{n - 1},f(X_{n - 2}))\in\mathfrak D_{Z}^{p/2}([0,\tau_0],\mathbb R^e)
\end{equation}
and
\begin{equation}\label{local_control_2_2}
\left\{
\begin{array}{rcl}
 \|X_n\|_{p\textrm{-var},\tau_0} & \leqslant & \mu_C\\
 \|H_n\|_{p\textrm{-var},\tau_0} & \leqslant & m_C\\
 \|Y_n\|_{1\textrm{-var},\tau_0} & \leqslant & M_C\\
 \|R_{X_n}\|_{p/2\textrm{-var},\tau_0} & \leqslant & M_R.
\end{array}\right.
\end{equation}
Set $X_1' := f(X_0)$. For every $(s,t)\in\Delta_{\tau_0}$,
\begin{eqnarray*}
 R_{X_1}(s,t) & = & 
 X_1(s,t) - X_1'(s)Z(s,t)\\
 & = &
 Y_1(s,t) +
 \int_{s}^{t}f(X_0(u))dZ(u) - f(X_0(s))Z(s,t).
\end{eqnarray*}
By Young-Love estimate (see Friz and Victoir \cite[Theorem 6.8]{FV10}, or \cite[Section 3.6 and the interesting historical notes pages 212-213]{DN11}), for every $(s,t)\in\Delta_{\tau_0}$,
\begin{displaymath}
\|R_{X_1}(s,t)\|\leqslant
\|Y_1\|_{p/2\textrm{-var},s,t} +
\frac{1}{1 - 2^{1-3/p}}\|f\|_{\textrm{Lip}^{\gamma}}\|Z\|_{p\textrm{-var},\tau_0}\|X_0\|_{p/2\textrm{-var},s,t}.
\end{displaymath}
By super-additivity of the control functions $\|Y_1\|_{p/2\textrm{-var},.}^{p/2}$ and $\|X_0\|_{p/2\textrm{-var},.}^{p/2}$, there exists a constant $c_1 > 0$, depending only on $p$ and $\|f\|_{\textrm{Lip}^{\gamma}}$, such that
\begin{eqnarray*}
 \|R_{X_1}\|_{p/2\textrm{-var},\tau_0}
 & \leqslant &
 c_1(\|Y_1\|_{p/2\textrm{-var},\tau_0}^{p/2} +
 \|Z\|_{p\textrm{-var},\tau_0}^{p/2}
 \|X_0\|_{p/2\textrm{-var},\tau_0}^{p/2})^{2/p}\\
 & \leqslant &
 c_1M_C(1 +
 \omega_{p,\mathbf Z}(0,\tau_0)^{1/2})^{2/p}.
\end{eqnarray*}
Then, $\|R_{X_1}\|_{p/2\textrm{-var},\tau_0}\leqslant c_1M_C4^{1/p}\leqslant M_R$ and
\begin{displaymath}
(X_1,f(X_0))\in\mathfrak D_{Z}^{p/2}([0,\tau_0],\mathbb R^e).
\end{displaymath}
So, the rough integral
\begin{displaymath}
H_2 :=
\int_{0}^{.}f(X_1(s))d\mathbf Z(s)
\end{displaymath}
is well defined. For every $(s,t)\in\Delta_T$,
\begin{eqnarray*}
 \|H_2(s,t)\| & \leqslant &
 \|f\|_{\textrm{Lip}^{\gamma}}\|Z\|_{p\textrm{-var},s,t} +
 \|f\|_{\textrm{Lip}^{\gamma}}^{2}\|\mathbb Z\|_{p/2\textrm{-var},s,t} +
 \mathfrak I_{f,\mathbf Z,X_n}(s,t)\\
 & \leqslant &
 (\|f\|_{\textrm{Lip}^{\gamma}}\vee
 \|f\|_{\textrm{Lip}^{\gamma}}^{2})\omega_{p,\mathbf Z}(s,t)^{1/p}\\
 & &
 + c(p)\|f\|_{\textrm{Lip}^{\gamma}}
 (1\vee\|f\|_{\textrm{Lip}^{\gamma}})(
 \|X_0\|_{p\textrm{-var},s,t} +\|X_1\|_{p\textrm{-var},s,t}\\
 & &
 + \|X_1\|_{p\textrm{-var},s,t}^{2} +\|R_{X_1}\|_{p/2\textrm{-var},s,t})\omega_{p,\mathbf Z}(s,t)^{1/p}\\
 & \leqslant &
 c_2(1 +\mu_C + M_R +\mu_{C}^{2})\omega_{p,\mathbf Z}(s,t)^{1/p},
\end{eqnarray*}
where $c_2 > 0$ is a constant depending only on $p$ and $\|f\|_{\textrm{Lip}^{\gamma}}$. By super-additivity of the control function $\omega_{p,\mathbf Z}$:
\begin{eqnarray*}
 \|H_2\|_{p\textrm{-var},\tau_0}
 & \leqslant &
 c_2(1 +\mu_C + M_R +\mu_{C}^{2})\omega_{p,\mathbf Z}(0,\tau_0)^{1/p}\\
 & \leqslant &
 m_C.
\end{eqnarray*}
So, by Proposition \ref{regularity_control_SD},
\begin{displaymath}
\|Y_2\|_{p\textrm{-var},\tau_0}
\leqslant
M_C
\end{displaymath}
and
\begin{displaymath}
\|X_2\|_{p\textrm{-var},\tau_0}
\leqslant
\|H_2\|_{p\textrm{-var},\tau_0} +\|Y_2\|_{p\textrm{-var},\tau_0}\\
\leqslant
\mu_C.
\end{displaymath}
Therefore, Conditions (\ref{local_control_2_1})-(\ref{local_control_2_2}) hold true for $n = 2$.
\\
\\
Assume that Conditions (\ref{local_control_2_1})-(\ref{local_control_2_2}) hold true until $n\in\mathbb N\backslash\{0,1\}$ arbitrarily chosen. Set $X_n' := f(X_{n - 1})$. For every $(s,t)\in\Delta_{\tau_0}$,
\begin{eqnarray*}
 R_{X_n}(s,t) & = & 
 X_n(s,t) - X_n'(s)Z(s,t)\\
 & = &
 Y_n(s,t) +
 \int_{s}^{t}f(X_{n - 1}(u))d\mathbf Z(u) - f(X_{n - 1}(s))Z(s,t).
\end{eqnarray*}
So, for every $(s,t)\in\Delta_{\tau_0}$,
\begin{eqnarray*}
 \|R_{X_n}(s,t)\| & \leqslant &
 \|Y_n(s,t)\| +\|Df(X_{n - 1}(s))f(X_{n - 2}(s))\mathbb Z(s,t)\| +
 \mathfrak I_{f,\mathbf Z,X_{n - 1}}(s,t)\\
 & \leqslant &
 \|Y_n\|_{p/2\textrm{-var},s,t} +
 \|f\|_{\textrm{Lip}^{\gamma}}^{2}\|\mathbb Z\|_{p/2\textrm{-var},s,t}\\
 & &
 + c(p)
 (\|Z\|_{p\textrm{-var},s,t}\|R_{f(X_{n - 1})}\|_{p/2\textrm{-var},s,t}\\
 & &
 + \|\mathbb Z\|_{p/2\textrm{-var},s,t}\|Df(X_{n - 1}(.))f(X_{n - 2})\|_{p\textrm{-var},s,t})\\
 & \leqslant &
 \|Y_n\|_{p/2\textrm{-var},s,t} +
 \|f\|_{\textrm{Lip}^{\gamma}}^{2}\|\mathbb Z\|_{p/2\textrm{-var},s,t}\\
 & &
 + c(p)\|f\|_{\textrm{Lip}^{\gamma}}
 (\|Z\|_{p\textrm{-var},\tau_0}\omega_n(s,t)^{2/p} +
 M(n,\tau_0)\|\mathbb Z\|_{p/2\textrm{-var},s,t}),
\end{eqnarray*}
where
\begin{eqnarray*}
 M(n,\tau_0) & := &
 \|f(X_{n - 2})\|_{p\textrm{-var},\tau_0} +\|f(X_{n - 2})\|_{\infty,\tau_0}\|X_{n - 1}\|_{p\textrm{-var},\tau_0}\\
 & \leqslant &
 \|f\|_{\textrm{Lip}^{\gamma}}(
 \|X_{n - 2}\|_{p\textrm{-var},\tau_0} +
 \|X_{n - 1}\|_{p\textrm{-var},\tau_0})
\end{eqnarray*}
and $\omega_n :\Delta_{\tau_0}\rightarrow\mathbb R_+$ is the control function defined by
\begin{displaymath}
\omega_n(u,v) :=
2^{p/2 - 1}(
\|X_{n - 1}\|_{p\textrm{-var},u,v}^{p} +\|R_{X_{n - 1}}\|_{p/2\textrm{-var},u,v}^{p/2})
\end{displaymath}
for every $(u,v)\in\Delta_{\tau_0}$. By super-additivity of the control functions
\begin{displaymath}
\|Y_n\|_{p/2\textrm{-var},.}^{p/2}
\textrm{, }\|\mathbb Z\|_{p/2\textrm{-var},.}^{p/2}
\textrm{ and }
\omega_n,
\end{displaymath}
there exist three constants $c_3,c_4,c_5 > 0$, depending only on $p$ and $\|f\|_{\textrm{Lip}^{\gamma}}$, such that
\begin{eqnarray*}
 \|R_{X_n}\|_{p/2\textrm{-var},\tau_0} & \leqslant &
 c_3
 (\|Y_n\|_{p/2\textrm{-var},\tau_0}^{p/2} +
 \|\mathbb Z\|_{p/2\textrm{-var},\tau_0}^{p/2}\\
 & &
 +\|Z\|_{p\textrm{-var},\tau_0}^{p/2}\omega_n(0,\tau_0) +
 M(n,\tau_0)^{p/2}\|\mathbb Z\|_{p/2\textrm{-var},\tau_0}^{p/2})^{2/p}\\
 & \leqslant &
 c_4(\|Y_n\|_{p/2\textrm{-var},\tau_0}^{p} +
 (1 +\omega_n(0,\tau_0) + M(n,\tau_0)^{p/2})^2\omega_{p,\mathbf Z}(0,\tau_0))^{1/p}\\
 & \leqslant &
 c_5(\mu_{C}^{p} + (1 +\mu_{C}^{p} +\|R_{X_{n - 1}}\|_{p/2\textrm{-var},\tau_0}^{p} +\mu_{C}^{2p})\omega_{p,\mathbf Z}(0,\tau_0))^{1/p}\\
 & \leqslant &
 c_5(\mu_{C}^{p} + (1 +\mu_{C}^{p} + M_{R}^{p} +\mu_{C}^{2p})\omega_{p,\mathbf Z}(0,\tau_0))^{1/p}.
\end{eqnarray*}
Then, $\|R_{X_n}\|_{p/2\textrm{-var},\tau_0}\leqslant c_5(\mu_{C}^{p} + 1)^{1/p}\leqslant M_R$ and
\begin{displaymath}
(X_n,f(X_{n - 1}))\in\mathfrak D_{Z}^{p/2}([0,\tau_0],\mathbb R^e).
\end{displaymath}
So, the rough integral
\begin{displaymath}
H_{n + 1} :=
\int_{0}^{.}f(X_n(s))d\mathbf Z(s)
\end{displaymath}
is well defined. For every $(s,t)\in\Delta_T$,
\begin{eqnarray*}
 \|H_{n + 1}(s,t)\| & \leqslant &
 \|f\|_{\textrm{Lip}^{\gamma}}\|Z\|_{p\textrm{-var},s,t} +
 \|f\|_{\textrm{Lip}^{\gamma}}^{2}\|\mathbb Z\|_{p/2\textrm{-var},s,t} +
 \mathfrak I_{f,\mathbf Z,X_n}(s,t)\\
 & \leqslant &
 (\|f\|_{\textrm{Lip}^{\gamma}}\vee
 \|f\|_{\textrm{Lip}^{\gamma}}^{2})\omega_{p,\mathbf Z}(s,t)^{1/p}\\
 & &
 + c(p)\|f\|_{\textrm{Lip}^{\gamma}}
 (1\vee\|f\|_{\textrm{Lip}^{\gamma}})(
 \|X_{n - 1}\|_{p\textrm{-var},s,t} +\|X_n\|_{p\textrm{-var},s,t}\\
 & &
 +\|X_n\|_{p\textrm{-var},s,t}^{2} +\|R_{X_n}\|_{p/2\textrm{-var},s,t})\omega_{p,\mathbf Z}(s,t)^{1/p}\\
 & \leqslant &
 c_6(1 +\mu_C + M_R +\mu_{C}^{2})\omega_{p,\mathbf Z}(s,t)^{1/p},
\end{eqnarray*}
where $c_6 > 0$ is a constant depending only on $p$ and $\|f\|_{\textrm{Lip}^{\gamma}}$. By super-additivity of the control function $\omega_{p,\mathbf Z}$:
\begin{eqnarray*}
 \|H_{n + 1}\|_{p\textrm{-var},\tau_0}
 & \leqslant &
 c_6(1 +\mu_C + M_R +\mu_{C}^{2})\omega_{p,\mathbf Z}(0,\tau_0)^{1/p}\\
 & \leqslant &
 m_C.
\end{eqnarray*}
So, by Proposition \ref{regularity_control_SD},
\begin{displaymath}
\|Y_{n + 1}\|_{p\textrm{-var},\tau_0}
\leqslant
M_C
\end{displaymath}
and
\begin{displaymath}
\|X_{n + 1}\|_{p\textrm{-var},\tau_0}
\leqslant
\|H_{n + 1}\|_{p\textrm{-var},\tau_0} +\|Y_{n + 1}\|_{p\textrm{-var},\tau_0}
\leqslant\mu_C.
\end{displaymath}
By induction, Conditions (\ref{local_control_2_1})-(\ref{local_control_2_2}) are satisfied for every $n\in\mathbb N\backslash\{0,1\}$. As in the proof of Theorem \ref{existence_Young}, the sequence $(H_n,X_n,Y_n)_{n\in\mathbb N\backslash\{0,1\}}$ is bounded in $\mathfrak C_{T}^{p,1}$. In addition, the sequence $(R_{X_n})_{n\in\mathbb N\backslash\{0,1\}}$ is bounded in $C^{p/2\textrm{-var}}([0,T],\mathbb R^e)$.
\\
\\
For every $n\in\mathbb N\backslash\{0,1\}$ and $(s,t)\in\Delta_T$,
\begin{eqnarray*}
 \|H_n(s,t)\| & \leqslant &
 (\|f\|_{\textrm{Lip}^{\gamma}}\vee
 \|f\|_{\textrm{Lip}^{\gamma}}^{2} \\
 & &
 + c(p)\|f\|_{\textrm{Lip}^{\gamma}}
 (1\vee\|f\|_{\textrm{Lip}^{\gamma}})
 (\sup_{n\in\mathbb N}
 \|X_{n - 2}\|_{p\textrm{-var},T} +\sup_{n\in\mathbb N}\|X_{n - 1}\|_{p\textrm{-var},T}\\
 & &
 +\sup_{n\in\mathbb N}\|X_{n - 1}\|_{p\textrm{-var},T}^{2} +
 \sup_{n\in\mathbb N}\|R_{X_{n - 1}}\|_{p/2\textrm{-var},T}))\omega_{p,\mathbf Z}(s,t)^{1/p}.
\end{eqnarray*}
Since $\omega_{p,\mathbf Z}$ is a control function, $(H_n)_{n\in\mathbb N\backslash\{0,1\}}$ is equicontinuous. Therefore, by Arzel\`a-Ascoli's theorem together with Proposition \ref{uniform_convergence_p_variation}, there exists an extraction $\varphi :\mathbb N\backslash\{0,1\}\rightarrow\mathbb N\backslash\{0,1\}$ such that $(H_{\varphi(n)})_{n\in\mathbb N\backslash\{0,1\}}$ converges uniformly to an element $H$ of $C^{p\textrm{-var}}([0,T],\mathbb R^e)$.
\\
\\
Since $(H_{\varphi(n)})_{n\in\mathbb N\backslash\{0,1\}}$ converges uniformly to $H$, by Theorem \ref{PRF_key_result}, the sequence $(X_{\varphi(n)},Y_{\varphi(n)})_{n\in\mathbb N\backslash\{0,1\}}$ converges uniformly to $(X,Y) := (v_H,w_H)$. So, for every $t\in [0,T]$,
\begin{displaymath}
\left\{
\begin{array}{rcl}
 X(t) & = & H(t) + Y(t)\\
 -\displaystyle{\frac{
 dDY}{d|DY|}}(t) & \in & N_{C_H(t)}(Y(t))
 \textrm{ $|DY|$-a.e. with }Y(0) = a,
\end{array}
\right.
\end{displaymath}
and by Proposition \ref{uniform_convergence_p_variation},
\begin{displaymath}
X\in C^{p\textrm{-var}}([0,T],\mathbb R^e)
\textrm{ and }
Y\in C^{1\textrm{-var}}([0,T],\mathbb R^e).
\end{displaymath}
Denoting $X' := f(X)$, $X'$ (resp.~$R_X$) is the uniform limit of $(X_{\varphi(n)}')_{n\in\mathbb N\backslash\{0,1\}}$ (resp.~$(R_{X_{\varphi(n)}})_{n\in\mathbb N\backslash\{0,1\}}$). So, by Proposition \ref{continuity_rough_integral_composition}:
\begin{displaymath}
\lim_{n\rightarrow\infty}
\left\|H_{\varphi(n)} -
\int_{0}^{.}f(X(s))d\mathbf Z(s)\right\|_{\infty,T} = 0.
\end{displaymath}
Therefore, since $(H_{\varphi(n)})_{n\in\mathbb N^*}$ converges also to $H$ in $C^0([0,T],\mathbb R^e)$,
\begin{displaymath}
H(t) =
\int_{0}^{t}f(X(s))d\mathbf Z(s)
\textrm{ $;$ }
\forall t\in [0,T].
\end{displaymath}
\end{proof}
%


%
\section{Some uniqueness results}
When $p = 1$ and there is an additive continuous signal of finite $q$-variation with $q\in [1,3[$, the uniqueness of the solution to Problem (\ref{rough_sweeping_process}) is established in Proposition
\ref{uniqueness_RS} below. Proposition \ref{SC_Young_uniqueness} and Proposition \ref{SC_rough_uniqueness} provide necessary and sufficient conditions for uniqueness of the solution when $p\in [1,2[$ and $p\in [2,3[$ respectively. These conditions are close to the monotonicity of the normal cone which allows to prove the uniqueness when $p = 1$ (see Proposition \ref{uniqueness_RS}).
%


%
\begin{proposition}\label{uniqueness_RS}
Assume that $p = 1$ and consider the Skorokhod problem
\begin{equation}\label{sweeping_process_additif}
\left\{
\begin{array}{rcl}
 X(t) & = & H(t) + Y(t)\\
 H(t) & = & \displaystyle{\int_{0}^{t}f(X(s))dZ(s)} + W(t)\\
 -\displaystyle{\frac{dDY}{d|DY|}(t)} & \in & N_{C_H(t)}(Y(t))\normalfont{\textrm{ $|DY|$-a.e. with $Y(0) = a$,}}
\end{array}\right.
\end{equation}
where $W\in C^{q\normalfont{\textrm{-var}}}([0,T],\mathbb R^e)$ with $q\in [1,3[$. Under Assumption \ref{assumption_C}, Problem (\ref{sweeping_process_additif}) has a unique solution which belongs to $C^{q\normalfont{\textrm{-var}}}([0,T],\mathbb R^e)$.
\end{proposition}
%


%
\begin{proof}
Consider two solutions $(X,Y)$ and $(X^*,Y^*)$ of Problem (\ref{rough_sweeping_process}) on $[0,T]$. Since $(s,t)\in\Delta_T\mapsto\|Z\|_{1\textrm{-var},s,t}$ is a control function, there exists $n\in\mathbb N^*$ and $(\tau_k)_{k\in\llbracket 0,n\rrbracket}\in\mathfrak D_{[0,T]}$ such that
\begin{equation}\label{control_uniqueness_RS}
\|Z\|_{1\textrm{-var},\tau_k,\tau_{k + 1}}\leqslant
M :=\frac{1}{4\|f\|_{\textrm{Lip}^{\gamma}}}
\textrm{ $;$ }
\forall k\in\llbracket 0,n - 1\rrbracket.
\end{equation}
For every $t\in [0,\tau_1]$,
\begin{eqnarray*}
 \|X(t) - X^*(t)\|^2 & = &
 \|H(t) - H^*(t)\|^2 +
 2\int_{0}^{t}\langle Y(s) - Y^*(s),d(Y - Y^*)(s)\rangle\\
 & &
 + 2\int_{0}^{t}\langle H(t) - H^*(t),d(Y - Y^*)(s)\rangle\\
 & \leqslant &
 m_1(\tau_1)^2 + 2m_2(t) + 2m_3(t),
\end{eqnarray*}
with $m_1(\tau_1) :=\|H - H^*\|_{\infty,\tau_1}$,
\begin{displaymath}
m_2(t) :=\int_{0}^{t}\langle X(s) - X^*(s),d(Y - Y^*)(s)\rangle,
\end{displaymath}
and
\begin{displaymath}
m_3(t) :=\int_{0}^{t}\langle H(t) - H^*(t) - (H(s) - H^*(s)),d(Y - Y^*)(s)\rangle.
\end{displaymath}
Consider $t\in [0,\tau_1]$. By Friz and Victoir \cite{FV10}, Proposition 2.2:
\begin{eqnarray*}
 \|H(t) - H^*(t)\|
 & = &
 \left\|\int_{0}^{t}(f(X(s)) - f(X^*(s)))dZ(s)\right\|\\
 & \leqslant &
 \|f\|_{\textrm{Lip}^{\gamma}}
 \|X - X^*\|_{\infty,\tau_1}\|Z\|_{1\textrm{-var},\tau_1}.
\end{eqnarray*}
So,
\begin{equation}\label{control_A}
 m_1(\tau_1)\leqslant
 \frac{1}{4}
 \|X - X^*\|_{\infty,\tau_1}.
\end{equation}
Since the map $x\in C(t)\mapsto N_{C(t)}(x)$ is monotone, $m_2(t)\leqslant 0$. By the integration by parts formula,
\begin{eqnarray*}
 m_3(t) & = &
 \int_{0}^{t}
 \langle Y(s) - Y^*(s),
 d(H - H^*)(s)\rangle\\
 & = &
 \int_{0}^{t}
 \langle X(s) - X^*(s) - (H(s) - H^*(s)),
 (f(X(s)) - f(X^*(s)))dZ(s)\rangle.
\end{eqnarray*}
So, by Friz and Victoir \cite{FV10}, Proposition 2.2 and Inequality (\ref{control_A}),
\begin{eqnarray*}
 m_3(t)
 & \leqslant &
 \|Df\|_{\infty}
 \|X - X^*\|_{\infty,\tau_1}
 (\|X - X^*\|_{\infty,\tau_1} +\|H - H^*\|_{\infty,\tau_1})\|Z\|_{1\textrm{-var},\tau_1}\\
 & \leqslant &
 \|f\|_{\textrm{Lip}^{\gamma}}\|Z\|_{1\textrm{-var},\tau_1}
 (1 +\|f\|_{\textrm{Lip}^{\gamma}}\|Z\|_{1\textrm{-var},\tau_1})\|X - X^*\|_{\infty,\tau_1}^{2}\\
 & \leqslant &
 5/16\|X - X^*\|_{\infty,\tau_1}^{2}.
\end{eqnarray*}
Therefore,
\begin{displaymath}
\|X - X^*\|_{\infty,\tau_1}^{2}
\leqslant
\frac{11}{16}\|X - X^*\|_{\infty,\tau_1}^{2}.
\end{displaymath}
Necessarily, $(X,Y) = (X^*,Y^*)$ on $[0,\tau_1]$.
\\
\\
For $k\in\llbracket 0,n - 1\rrbracket$, assume that $(X,Y) = (X^*,Y^*)$ on $[0,\tau_k]$.
By Equation (\ref{control_uniqueness_RS}) and exactly the same ideas as on $[0,\tau_1]$:
\begin{displaymath}
\|X - X^*\|_{\infty,\tau_k,\tau_{k + 1}}^{2}
\leqslant
\frac{11}{16}\|X - X^*\|_{\infty,\tau_k,\tau_{k + 1}}^{2}.
\end{displaymath}
So, $(X,Y) = (X^*,Y^*)$ on $[0,\tau_{k + 1}]$. Recursively, Problem (\ref{rough_sweeping_process}) has a unique solution on $[0,T]$.
\end{proof}
\noindent
\textbf{Remark.} The cornerstone of the proof of Proposition \ref{uniqueness_RS} is that
\begin{equation}\label{cornerstone_uniqueness_RS}
\int_{0}^{t}\langle X(s) - X^*(s),d(Y - Y^*)(s)\rangle
\leqslant 0
\textrm{ $;$ }
\forall t\in [0,T].
\end{equation}
Thanks to the monotonicity of the map $x\in C(t)\mapsto N_{C(t)}(x)$
($t\in [0,T]$), Inequality (\ref{cornerstone_uniqueness_RS}) is
true. When $p\in ]1,3[$, it is not possible to get inequalities
involving only the uniform norm of $X - X^*$.
In that case, the construction of the Young/rough integral suggests to
use ideas similar to those of the proof of Proposition
\ref{uniqueness_RS}, but using the $p$-variation norm of $X - X^*$. 
\\
\\
In a probabilistic setting, uniqueness up to equality almost everywhere
can be obtained for Brownian motion, with $p>2$, 
in the frame of It\^o calculus, using the martingale property of
stochastic integrals and Doob's inequality, see 
\cite{TANAKA,LS84,SAISHO} for a fixed convex set $C$ and 
\cite{BV11,CMR15} for a moving set. 
\\
\\
The two following propositions show that when $p\in ]1,3[$, there exist
some conditions close to Inequality (\ref{cornerstone_uniqueness_RS}),
ensuring the uniqueness of the solution to Problem
(\ref{rough_sweeping_process}).

\begin{proposition}\label{SC_Young_uniqueness}
Consider $(s,t)\in\Delta_T$, $p\in [1,2[$ and two solutions $(X,Y)$
and $(X^*,Y^*)$ to Problem (\ref{rough_sweeping_process}) under
Assumption \ref{assumption_C}. On $[s,t]$, $(X,Y) = (X^*,Y^*)$ if and only if $X(s) = X^*(s)$ and
\begin{equation}\label{eq:inegfond}
\int_{u}^{v}
\langle X(u,r) - X^*(u,r),d(Y - Y^*)(r)\rangle
\leqslant 0
\textrm{ $;$ }
\forall (u,v)\in\Delta_{s,t}.
\end{equation}
\end{proposition}
%


%
\begin{proof}
For the sake of simplicity, the proposition is proved on $[0,T]$ instead of
$[s,t]$, with $(s,t)\in\Delta_T$. 
\\
\\
First of all, if $(X,Y) = (X^*,Y^*)$ on $[s,t]$, then
\begin{displaymath}
\int_{u}^{v}
\langle X(u,r) - X^*(u,r),d(Y - Y^*)(r)\rangle = 0
\textrm{ $;$ }
\forall (u,v)\in\Delta_{s,t}.
\end{displaymath}
Now, let us prove that if $X(s) = X^*(s)$ and Inequality (\ref{eq:inegfond}) is true, then $(X,Y) = (X^*,Y^*)$.
\\
\\
For every $(s,t)\in\Delta_T$,
\begin{eqnarray*}
 \|X(s,t) - X^*(s,t)\|^2 & = &
 \|H(s,t) - H^*(s,t)\|^2\\
 & &
 + 2\int_{s}^{t}\langle Y(s,u) - Y^*(s,u),d(Y - Y^*)(u)\rangle\\
 & &
 + 2 \int_{s}^{t}\langle H(s,t) - H^*(s,t),d(Y - Y^*)(u)\rangle\\
 & = &
 \|H(s,t) - H^*(s,t)\|^2\\
 & &
 + 2\int_{s}^{t}\langle X(s,u) - X^*(s,u),d(Y - Y^*)(u)\rangle\\
 & &
 + 2\int_{s}^{t}\langle H(s,t) - H(s,u) - (H^*(s,t) - H^*(s,u)),d(Y - Y^*)(u)\rangle\\
 & \leqslant &
 \|H - H^*\|_{p\textrm{-var},s,t}^{2} + 2m(s,t)
\end{eqnarray*}
with
\begin{displaymath}
m(s,t) :=
\int_{s}^{t}\langle H(u,t) - H^*(u,t),d(Y - Y^*)(u)\rangle.
\end{displaymath}
Let $(s,t)\in\Delta_T$ be arbitrarily chosen.
\\
\\
On the one hand,
\begin{displaymath}
m(s,t)\leqslant
2e\cdot c(p,p)\|H - H^*\|_{p\textrm{-var},s,t}
\|Y - Y^*\|_{p\textrm{-var},s,t}.
\end{displaymath}
So, there exists a constant $c_1 > 0$, not depending on $s$ and $t$, such that
\begin{eqnarray*}
 \|X(s,t) - X^*(s,t)\|^p
 & \leqslant &
 (\|H - H^*\|_{p\textrm{-var},s,t}^{2} + 2m(s,t))^{p/2}\\
 & \leqslant &
 c_1(\|H - H^*\|_{p\textrm{-var},s,t}^{p}\\
 & &
 + (\|H - H^*\|_{p\textrm{-var},s,t}^{p})^{1/2}
 (\|Y - Y^*\|_{p\textrm{-var},s,t}^{p})^{1/2}).
\end{eqnarray*}
Since $1/2 + 1/2  = 1$, the right-hand side of the previous inequality defines a control function (see Friz and Victoir \cite{FV10}, Exercice 1.9), and then there exists a constant $c_2 > 0$, not depending on $s$ and $t$, such that
\begin{eqnarray}
 \|X - X^*\|_{p\textrm{-var},s,t}^{p}
 & \leqslant &
 c_1(\|H - H^*\|_{p\textrm{-var},s,t}^{p} +
 \|H - H^*\|_{p\textrm{-var},s,t}^{p/2}
 \|Y - Y^*\|_{p\textrm{-var},s,t}^{p/2})
 \nonumber\\
 \label{control_couple_solutions_1}
 & \leqslant &
 c_2(\|H - H^*\|_{p\textrm{-var},s,t}^{p} +\|H - H^*\|_{p\textrm{-var},s,t}^{p/2}\|X - X^*\|_{p\textrm{-var},s,t}^{p/2}).
\end{eqnarray}
The right-hand side of the previous inequality defines a control function.
\\
\\
On the other hand, since $X(0) = X^*(0)$:
\begin{eqnarray*}
 \|H - H^*\|_{p\textrm{-var},s,t} & \leqslant &
 c(p,p)\|Z\|_{p\textrm{-var},s,t}(\|f\circ X - f\circ X^*\|_{p\textrm{-var},s,t}\\
 & &
 +\|f(X(s)) - f(X(0)) - (f(X^*(s)) - f(X^*(0)))\|)\\
 & \leqslant &
 2c(p,p)\|Z\|_{p\textrm{-var},s,t}\|f\circ X - f\circ X^*\|_{p\textrm{-var},t}.
\end{eqnarray*}
Consider $(u,v)\in\Delta_t$ and
\begin{displaymath}
\delta(u,v) :=
\|(f\circ X)(u,v) - (f\circ X^*)(u,v)\|.
\end{displaymath}
Applying Taylor's formula to the map $f$ between $X(u,v)$ and \mbox{$X^*(u,v)$:}
\begin{eqnarray*}
 \delta(u,v) & \leqslant &
 \left\|\int_{0}^{1}Df(X(u) +\theta X(u,v))(X(u,v) - X^*(u,v))d\theta\right\|\\
 & &
 +\left\|\int_{0}^{1}(Df(X(u) +\theta X(u,v)) - Df(X^*(u) +\theta X^*(u,v))) X^*(u,v)d\theta\right\|\\
 & \leqslant &
 \|f\|_{\textrm{Lip}^{\gamma}}(
 \|X - X^*\|_{p\textrm{-var},u,v} +
 2\|X^*\|_{p\textrm{-var},u,v}
 \|X - X^*\|_{p\textrm{-var},t}).
\end{eqnarray*}
So, there exists a constant $c_3 > 0$, not depending on $t$, such that
\begin{displaymath}
\|f\circ X - f\circ X^*\|_{p\textrm{-var},t}
\leqslant
c_3\|X - X^*\|_{p\textrm{-var},t},
\end{displaymath}
and then there exists a constant $c_4 > 0$, not depending on $s$ and $t$, such that
\begin{equation}\label{control_couple_solutions_2}
\|H - H^*\|_{p\textrm{-var},s,t}
\leqslant
c_4\|Z\|_{p\textrm{-var},s,t}\|X - X^*\|_{p\textrm{-var},t}.
\end{equation}
By Equation (\ref{control_couple_solutions_1}) and Equation (\ref{control_couple_solutions_2}) together, there exists a constant $c_5 > 0$, not depending on $s$ and $t$, such that
\begin{displaymath}
\|X - X^*\|_{p\textrm{-var},s,t}\leqslant
c_5\|Z\|_{p\textrm{-var},s,t}^{1/2}\|X - X^*\|_{p\textrm{-var},t}.
\end{displaymath}
Since $(u,v)\in\Delta_T\mapsto\|Z\|_{p\textrm{-var},u,v}^{p}$ is a control function, there exists $N\in\mathbb N^*$ and $(\tau_k)_{k\in\llbracket 0,N\rrbracket}\in\mathfrak D_{[0,T]}$ such that
\begin{displaymath}
\|Z\|_{p\textrm{-var},\tau_k,\tau_{k + 1}}
\leqslant\frac{1}{4c_{5}^{2}}
\textrm{ $;$ }
\forall k\in\llbracket 0,N - 1\rrbracket.
\end{displaymath}
First,
\begin{eqnarray*}
 \|X - X^*\|_{p\textrm{-var},\tau_1}
 & \leqslant &
 c_5\|Z\|_{p\textrm{-var},\tau_1}^{1/2}\|X - X^*\|_{p\textrm{-var},\tau_1}\\
 & \leqslant &
 \frac{1}{2}\|X - X^*\|_{p\textrm{-var},\tau_1}.
\end{eqnarray*}
So, $X = X^*$ on $[0,\tau_1]$. For $k\in\llbracket 1,N - 1\rrbracket$, assume that $X = X^*$ on $[0,\tau_k]$. Then,
\begin{eqnarray*}
 \|X - X^*\|_{p\textrm{-var},\tau_{k + 1}} & = &
 \|X - X^*\|_{p\textrm{-var},\tau_k,\tau_{k + 1}}\\
 & \leqslant &
 \frac{1}{2}\|X - X^*\|_{p\textrm{-var},\tau_{k + 1}}.
\end{eqnarray*}
So, $X = X^*$ on $[0,\tau_{k + 1}]$. Recursively, $X = X^*$ on $[0,T]$.
\end{proof}

\begin{proposition}\label{SC_rough_uniqueness}
Consider $(s,t)\in\Delta_T$, $p\in [2,3[$ and two solutions $(X,Y)$
and $(X^*,Y^*)$ to Problem (\ref{rough_sweeping_process}) under
Assumption \ref{assumption_C}. On $[s,t]$, $(X,Y) = (X^*,Y^*)$ if and only if $X(s) = X^*(s)$ and
\begin{equation}\label{eq:inegfondrough}
\int_{u}^{v}
\langle R_X(u,r) - R_{X^*}(u,r),d(Y - Y^*)(r)\rangle
\leqslant 0
\textrm{ $;$ }
\forall (u,v)\in\Delta_{s,t}.
\end{equation}
\end{proposition}
%


%
\begin{proof}
For the sake of simplicity, the proposition is proved on $[0,T]$ instead of $[s,t]$ with $(s,t)\in\Delta_T$.
\\
\\
First of all, if $(X,Y) = (X^*,Y^*)$ on $[s,t]$, then
\begin{displaymath}
\int_{u}^{v}
\langle R_X(u,r) - R_{X^*}(u,r),d(Y - Y^*)(r)\rangle = 0
\textrm{ $;$ }
\forall (u,v)\in\Delta_{s,t}.
\end{displaymath}
Now, let us prove that if $X(s) = X^*(s)$ and Inequality (\ref{eq:inegfondrough}) is true, then $(X,Y) = (X^*,Y^*)$.
\\
\\
There exists a constant $c_1 > 0$ such that for every $(s,t)\in\Delta_T$,
\begin{multline}\label{SC_rough_uniqueness_1}
  \|(X - X^*,(X - X^*)')\|_{Z,p/2,s,t}\\  
  \begin{aligned}
 =&\|f(X) - f(X^*)\|_{p\textrm{-var},s,t} +
 \|R_X - R_{X^*}\|_{p/2\textrm{-var},s,t}
 \\
  \leqslant& 
 c_1(\|R_X - R_{X^*}\|_{p/2\textrm{-var},s,t} 
 + \|Z\|_{p\textrm{-var},s,t}\|(X - X^*,(X - X^*)')\|_{Z,p/2,t}).
 \end{aligned}
\end{multline}
Let us find a suitable control function dominating
\begin{displaymath}
(s,t)\in\Delta_T
\longmapsto
\|R_X - R_{X^*}\|_{p/2\textrm{-var},s,t}^{p/2}.
\end{displaymath}
For every $(s,t)\in\Delta_T$,
\begin{eqnarray*}
 \|R_X(s,t) - R_{X^*}(s,t)\|^2 & = &
 \|R_H(s,t) - R_{H^*}(s,t)\|^2\\
 & &
 + 2\int_{s}^{t}\langle Y(s,u) - Y^*(s,u),d(Y - Y^*)(u)\rangle\\
 & &
 + 2\int_{s}^{t}\langle R_H(s,t) - R_{H^*}(s,t),d(Y - Y^*)(u)\rangle\\
 & = &
 \|R_H(s,t) - R_{H^*}(s,t)\|^2 \\
 & &
 + 2\int_{s}^{t}\langle R_X(s,u) - R_{X^*}(s,u),d(Y - Y^*)(u)\rangle\\
 & &
 + 2\int_{s}^{t}\langle R_H(s,t) - R_H(s,u) - (R_{H^*}(s,t) - R_{H^*}(s,u)),d(Y - Y^*)(u)\rangle\\
 & \leqslant &
 \|R_H - R_{H^*}\|_{p/2\textrm{-var},s,t}^{2} + 2m(s,t)
\end{eqnarray*}
with
\begin{displaymath}
m(s,t) :=
\int_{s}^{t}\langle R_H(u,t) - R_{H^*}(u,t),d(Y - Y^*)(u)\rangle.
\end{displaymath}
Let $(s,t)\in\Delta_T$ be arbitrarily chosen.
\\
\\
On the one hand,
\begin{displaymath}
m(s,t)\leqslant
2e\cdot c(p,p)\|R_H - R_{H^*}\|_{p/2\textrm{-var},s,t}
\|Y - Y^*\|_{p/2\textrm{-var},s,t}.
\end{displaymath}
So, there exists a constant $c_2 > 0$, not depending on $s$ and $t$, such that
\begin{eqnarray*}
 \|R_X(s,t) - R_{X^*}(s,t)\|^{p/2}
 & \leqslant &
 (\|R_H - R_{H^*}\|_{p/2\textrm{-var},s,t}^{2} + 2m(s,t))^{p/4}\\
 & \leqslant &
 c_2(\|R_H - R_{H^*}\|_{p/2\textrm{-var},s,t}^{p/2}\\
 & &
 + (\|R_H - R_{H^*}\|_{p/2\textrm{-var},s,t}^{p/2})^{1/2}
 (\|Y - Y^*\|_{p/2\textrm{-var},s,t}^{p/2})^{1/2}).
\end{eqnarray*}
Since $1/2 + 1/2  = 1$, the right-hand side of the previous inequality defines a control function (see Friz and Victoir \cite{FV10}, Exercice 1.9), and then there exists a constant $c_3 > 0$, not depending on $s$ and $t$, such that
\begin{eqnarray}
 \|R_X - R_{X^*}\|_{p/2\textrm{-var},s,t}^{p/2}
 & \leqslant &
 c_2(\|R_H - R_{H^*}\|_{p/2\textrm{-var},s,t}^{p/2} 
 \nonumber\\
 & &
 +\|R_H - R_{H^*}\|_{p/2\textrm{-var},s,t}^{p/4}
 \|Y - Y^*\|_{p/2\textrm{-var},s,t}^{p/4})
 \nonumber\\
 & \leqslant &
 c_3(\|R_H - R_{H^*}\|_{p/2\textrm{-var},s,t}^{p/2} 
 \nonumber\\
 \label{SC_rough_uniqueness_2}
 & &
 +\|R_H - R_{H^*}\|_{p/2\textrm{-var},s,t}^{p/4}
 \|R_X - R_{X^*}\|_{p/2\textrm{-var},s,t}^{p/4}).
\end{eqnarray}
On the other hand, since $X(0) = X^*(0)$:
\begin{eqnarray*}
 \|f(X)'(s) - f(X^*)'(s)\| & = &
 \|Df(X(s))f(X(s)) - Df(X(0))f(X(0))\\
 & &
 - (Df(X^*(s))f(X^*(s)) - Df(X^*(0))f(X^*(0)))\|\\
 & \leqslant &
 \|f(X)' - f(X^*)'\|_{p\textrm{-var},t}.
\end{eqnarray*}
Then,
\begin{eqnarray*}
 \|R_H(s,t) - R_{H^*}(s,t)\| & \leqslant &
 \mathfrak I_{\mathbf Z,f(X) - f(X^*)}(s,t) +
 \|(f(X)'(s) - f(X^*)'(s))\mathbb Z(s,t)\|\\
 & \leqslant &
 c(p)(\|R_{f(X)} - R_{f(X^*)}\|_{p/2\textrm{-var},s,t}\|Z\|_{p\textrm{-var},s,t}\\
 & &
 +\|f(X)' - f(X^*)'\|_{p\textrm{-var},s,t}\|\mathbb Z\|_{p/2\textrm{-var},s,t})\\
 & &
 +\|f(X)' - f(X^*)'\|_{p\textrm{-var},t}\|\mathbb Z\|_{p/2\textrm{-var},s,t}.
\end{eqnarray*}
So, with the same ideas as in P.~Friz and M.~Hairer \cite[Theorem 8.4
p. 115]{FH14},
there exists a constant $c_4 > 0$, not depending on $s$ and $t$, such that
\begin{equation}\label{SC_rough_uniqueness_3}
\|R_H - R_{H^*}\|_{p/2\textrm{-var},s,t}\leqslant
c_4\|(X - X^*,(X - X^*)')\|_{Z,p/2,t}\omega_{p,\mathbf Z}(s,t)^{1/p}.
\end{equation}
By Equations (\ref{SC_rough_uniqueness_1}), (\ref{SC_rough_uniqueness_2}) and (\ref{SC_rough_uniqueness_3}) together, there exists a constant $c_5 > 0$, not depending on $s$ and $t$, such that
\begin{displaymath}
\|(X - X^*,(X - X^*)')\|_{Z,p/2,s,t}
\leqslant
c_5\|(X - X^*,(X - X^*)')\|_{Z,p/2,t}\omega_{p,\mathbf Z}(s,t)^{1/(2p)}.
\end{displaymath}
The conclusion of the proof is the same as in Proposition \ref{SC_Young_uniqueness}.
\end{proof}
%


%
\section{Approximation scheme}
In Proposition \ref{uniqueness_RS}, it has been proved that, under
Assumption \ref{assumption_C}, Problem (\ref{rough_sweeping_process})
has a unique solution $(X,Y)$ if $p = 1$ and if,
  moreover, there is an additive continuous signal of finite
$q$-variation $W$ with $q\in [1,3[$. This section deals with the
convergence of the following approximation scheme for $X$: 
\begin{equation}\label{approximation_skorokhod_problem}
\left\{
\begin{array}{rcl}
 X_{0}^{n} & := & a\\
 X_{k + 1}^{n} & = & p_{C(t_{k + 1}^{n})}(X_{k}^{n} + f(X_{k}^{n})(t_{k + 1}^{n} - t_{k}^{n}) + W(t_{k}^{n},t_{k + 1}^{n}))\textrm{ $;$ } k\in\llbracket 0,n - 1\rrbracket,
\end{array}\right.
\end{equation}
where $n\in\mathbb N^*$ and $(t_{0}^{n},\dots,t_{n}^{n})$ is the dissection of $[0,T]$ of constant mesh $T/n$.
\\
\\
Consider the maps $X^n$, $H^n$ and $Y^n$ from $[0,T]$ into $\mathbb R^e$ defined by $X^n(t) := X_{k}^{n}$,
\begin{equation}\label{Hn_definition}
H^n(t) :=
\sum_{i = 0}^{k - 1}
f(X_{i}^{n})(t_{i + 1}^{n} - t_{i}^{n}) +
f(X_{k}^{n})(t - t_{k}^{n}) +
W(t)
\end{equation}
and
\begin{equation}\label{XHYn_relation}
Y^n(t) := X^n(t) - H^n(t_{k}^{n})
\end{equation}
for every $k\in\llbracket 0,n - 1\rrbracket$ and $t\in [t_{k}^{n},t_{k + 1}^{n}[$.
%


%
\begin{lemma}\label{convergence_Hn}
Under Assumption \ref{assumption_C}, one can extract a uniformly converging subsequence from any subsequence of $(H^n)_{n\in\mathbb N^*}$.
\end{lemma}
%


%
\begin{proof}
On the one hand, since $C(t)$ is a bounded set for every $t\in [0,T]$, $C$ is continuous on $[0,T]$ for the Hausdorff distance and $X^n([0,T])\subset\cup_{t\in [0,T]}C(t)$ for every $n\in\mathbb N^*$ by construction,
\begin{displaymath}
\sup_{n\in\mathbb N^*}
\|X^n\|_{\infty,T} <\infty.
\end{displaymath}
On the other hand, consider $(s,t)\in\Delta_T$ and $j,k\in\llbracket 0,n\rrbracket$ such that $s < t_{i}^{n}\leqslant t$ for every $i\in\llbracket j,k\rrbracket$. Then,
\begin{eqnarray*}
 \|H^n(t) - H^n(s)\| & = &
 \left\|
 \sum_{i = 0}^{k - 1}f(X_{i}^{n})(t_{i + 1}^{n} - t_{i}^{n}) + f(X_{k}^{n})(t - t_{k}^{n})
 \right.\\
 & &
 \left.
 -\sum_{i = 0}^{j - 2}f(X_{i}^{n})(t_{i + 1}^{n} - t_{i}^{n}) - f(X_{j - 1}^{n})(s - t_{j}^{n}) +
 W(s,t)\right\|\\
 & = &
 \left\|\int_{s}^{t}f(X^n(u))du + W(s,t)\right\|\\
 & \leqslant &
 \varphi(s,t) :=
 |t - s|\sup_{n\in\mathbb N^*}\|f\circ X^n\|_{\infty,T} +\|W\|_{q\textrm{-var},s,t}.
\end{eqnarray*}
Since $(s,t)\in\Delta_T\mapsto\varphi(s,t)$ is a continuous map such that $\varphi(t,t) = 0$ for every $t\in [0,T]$, $(H^n)_{n\in\mathbb N^*}$ is equicontinuous. Therefore, by Arzel\`a-Ascoli's theorem, one can extract a uniformly converging subsequence from any subsequence of $(H^n)_{n\in\mathbb N^*}$.
\end{proof}
%


%
\begin{lemma}\label{control_variation_Yn}
Under Assumption \ref{assumption_C}, there exist $R > 0$ and $N\in\mathbb N^*$ such that
\begin{displaymath}
\sup_{n\in\mathbb N^*}
\|Y^n\|_{1\normalfont{\textrm{-var}},T}
\leqslant M(N,R)
\end{displaymath}
with
\begin{displaymath}
M(N,R) :=
\frac{N}{R}\left(
\|\gamma\|_{\infty,T} +
\sup_{n\in\mathbb N^*}\|X^n\|_{\infty,T} +
\varphi(0,T)\right)^2.
\end{displaymath}
\end{lemma}
%


%
\begin{proof}
On the one hand, since the map $(s,t)\in\Delta_T\mapsto\varphi(s,t)$ defined in the proof of Lemma \ref{control_variation_Yn} is continuous and satisfies $\varphi(t,t) = 0$ for every $t\in [0,T]$, by Assumption \ref{assumption_C}, there exist $R > 0$, $N\in\mathbb N^*$ and a dissection $(\tau_0,\dots,\tau_N)$ of $[0,T]$ such that
\begin{displaymath}
\overline B_e(\gamma(\tau_i),R)\subset C(u)
\textrm{ and }
\varphi(\tau_i,\tau_{i + 1})\leqslant R/2
\end{displaymath}
for every $i\in\llbracket 0,N - 1\rrbracket$ and $u\in [\tau_i,\tau_{i + 1}[$. Then,
\begin{displaymath}
\overline B_e(\gamma(\tau_i) - H^n(\tau_i),R/2)\subset
\overline B_e(\gamma(\tau_i) - H^n(u),R)\subset C(u) - H^n(u)
\end{displaymath}
for every $i\in\llbracket 0,N - 1\rrbracket$ and $u\in [\tau_i,\tau_{i + 1}[$.
\\
\\
On the other hand, for every $k\in\llbracket 1,n\rrbracket$,
\begin{eqnarray*}
 Y^n(t_{k}^{n}) & = &
 p_{C(t_{k}^{n})}(X_{k - 1}^{n} + H^n(t_{k - 1}^{n},t_{k}^{n})) - H^n(t_{k}^{n})\\
 & = &
 p_{C(t_{k}^{n}) - H^n(t_{k}^{n})}(Y^n(t_{k - 1}^{n})).
\end{eqnarray*}
So, for any $i\in\llbracket 0,N - 1\rrbracket$, by applying Proposition \ref{approximation_scheme_SP}.(1) to $Y^n$ on $[\tau_i,\tau_{i + 1}[$:
\begin{eqnarray*}
 \|Y^n\|_{1\textrm{-var},\tau_{i},\tau_{i + 1}}
 & \leqslant &
 l(R/2,\|\gamma(\tau_i) - H^n(\tau_i) - Y^n(\tau_i)\|)\\
 & \leqslant &
 R^{-1}\|\gamma(\tau_i) - H^n(\tau_i) - Y^n(\tau_i)\|^2.
\end{eqnarray*}
Since there exists $j\in\llbracket 0,n\rrbracket$ such that $Y^n(\tau_i) = Y^n(t_{j}^{n})$,
\begin{eqnarray*}
 \|Y^n\|_{1\textrm{-var},\tau_{i},\tau_{i + 1}}
 & \leqslant &
 R^{-1}(\|\gamma(\tau_i)\| +\|H^n(\tau_i) - H^n(t_{j}^{n})\| +\|H^n(t_{j}^{n}) + Y^n(t_{j}^{n})\|)^2\\
 & \leqslant &
 R^{-1}(\|\gamma(\tau_i)\| +\varphi(t_{j}^{n},\tau_i) +\|X_{j}^{n}\|)^2\\
 & \leqslant &
 N^{-1}M(N,R).
\end{eqnarray*}
Therefore,
\begin{displaymath}
\|Y^n\|_{1\textrm{-var},T} =
\sum_{i = 0}^{N - 1}
\|Y^n\|_{1\textrm{-var},\tau_{i},\tau_{i + 1}}\\
\leqslant
M(N,R).
\end{displaymath}
\end{proof}
%


%
\begin{lemma}\label{inequality_Yn}
Under Assumption \ref{assumption_C}, for every $(s,t)\in\Delta_T$ and $z\in\cap_{\tau\in [s,t]}(C(\tau) - H^n(\tau))$,
\begin{displaymath}
\langle z,Y^n(t) - Y^n(s)\rangle
\geqslant\frac{1}{2}(\|Y^n(t)\|^2 -\|Y^n(s)\|^2).
\end{displaymath}
\end{lemma}
%


%
\begin{proof}
Consider $(s,t)\in\Delta_T$. There exists a maximal interval $\llbracket j,k\rrbracket\subset\llbracket 0,n\rrbracket$ such that
\begin{displaymath}
s < t_{i}^{n}\leqslant t
\textrm{ $;$ }
\forall i\in\llbracket j,k\rrbracket.
\end{displaymath}
Consider $z\in\cap_{\tau\in [s,t]}(C(\tau) - H^n(\tau))$. In particular, for every $i\in\llbracket j,k\rrbracket$, there exists $y_i\in C(t_{i}^{n})$ such that $z = y_i - H^n(t_{i}^{n})$. For every $i\in\llbracket j,k\rrbracket$,
\begin{eqnarray*}
 \langle z - Y^n(t_{i}^{n}),
 Y^n(t_{i}^{n}) - Y^n(t_{i - 1}^{n})\rangle
 & = &\\
 \langle y_i - H^n(t_{i}^{n}) - Y^n(t_{i}^{n}),
 Y^n(t_{i}^{n}) - Y^n(t_{i - 1}^{n})\rangle
 & = &\\
 \langle y_i - X_{i}^{n},X_{i}^{n} - (X_{i - 1}^{n} + H^n(t_{i - 1}^{n},t_{i}^{n}))\rangle
 & \geqslant & 0
\end{eqnarray*}
because
\begin{displaymath}
X_{i}^{n} =
p_{C(t_{i}^{n})}(X_{i - 1}^{n} + H^n(t_{i - 1}^{n},t_{i}^{n})).
\end{displaymath}
Then,
\begin{eqnarray*}
 \langle z,Y^n(t) - Y^n(s)\rangle
 & = &
 \langle z,Y^n(t_{k}^{n}) - Y^n(t_{j - 1}^{n})\rangle
 =\sum_{i = j}^{k}\langle z,Y^n(t_{i}^{n}) - Y^n(t_{i - 1}^{n})\rangle\\
 & \geqslant &
 \sum_{i = j}^{k}\langle Y^n(t_{i}^{n}),Y^n(t_{i}^{n}) - Y^n(t_{i - 1}^{n})\rangle\\
 & \geqslant &
 \frac{1}{2}\sum_{i = j}^{k}
 (\|Y^n(t_{i}^{n})\|^2 -\|Y^n(t_{i - 1}^{n})\|^2) =\frac{1}{2}(\|Y^n(t)\|^2 -\|Y^n(s)\|^2).
\end{eqnarray*}
\end{proof}
\noindent
Now, $W$ is $1/q$-H\"older continuous from $[0,T]$ into $\mathbb R^e$. Then, there exists a constant $C_{\varphi} > 0$ such that
\begin{displaymath}
\varphi(s,t)
\leqslant C_{\varphi}|t - s|^{1/q}
\textrm{ $;$ }
\forall (s,t)\in\Delta_T.
\end{displaymath}
%


%
\begin{theorem}\label{convergence_subsequences}
Assume that $C$ fulfills Assumption \ref{assumption_C} and that there exists $(K,\alpha)\in ]0,\infty[\times ]0,1[$ such that
\begin{displaymath}
d_H(C(s),C(t))
\leqslant
K|t - s|^{\alpha}
\textrm{ $;$ }
\forall (s,t)\in\Delta_T.
\end{displaymath}
%
Then, $(X^{n},Y^{n})_{n\in\mathbb N^*}$
converges uniformly to the unique solution $(X,Y)$
to Problem (\ref{rough_sweeping_process}).
\end{theorem}
%


%
\begin{proof}
Consider an extraction $\psi :\mathbb N^*\rightarrow\mathbb N^*$
such that $(H^{\psi(n)})_{n\in\mathbb N^*}$ is uniformly converging
to a limit $H^*$.
\\
\\
On the one hand, consider $n\in\mathbb N^*$ such that $T/n\in ]0,1]$, $m\in n\mathbb N^*$ and $t\in [0,T]$. By Proposition \ref{approximation_scheme_SP}.(2) together with Lemma \ref{control_variation_Yn}, there exist $R > 0$, $N\in\mathbb N^*$, $i\in\llbracket 1,n\rrbracket$ and $j\in\llbracket 1,m\rrbracket$ such that $t\in [t_{i - 1}^{n},t_{i}^{n}[$, $t\in [t_{j - 1}^{m},t_{j}^{m}[$ and
\begin{eqnarray}
 \|Y^n(t) - Y^m(t)\|^2
 & \leqslant & 2d_H(C(t_{i}^{n}) - H^n(t_{i}^{n}),C(t_{j}^{m}) - H^m(t_{j}^{m}))
 \nonumber\\
 & &
 \times
 (\|Y^n\|_{1\normalfont{\textrm{-var}},T} +
 \|Y^m\|_{1\normalfont{\textrm{-var}},T})
 \nonumber\\
 & \leqslant &
 4M(N,R)(d_H(C(t_{i}^{n}),C(t_{j}^{m})) +\|H^n(t_{i}^{n}) - H^m(t_{j}^{m})\|)
 \nonumber\\
 & \leqslant &
 4M(N,R)(K|t_{i}^{n} - t_{j}^{m}|^{\alpha}
 \nonumber\\
 & &
 +\|H^n(t_{i}^{n}) - H^n(t_{j}^{m})\| +
 \|H^n(t_{j}^{m}) - H^m(t_{j}^{m})\|)
 \nonumber\\
 \label{convergence_subsequences_1}
 & \leqslant &
 4M(N,R)((C_{\varphi} + K)|T/n|^{\alpha\wedge 1/q}
 +\|H^n - H^m\|_{\infty,T}).
\end{eqnarray}
Consider $\varepsilon > 0$. There exists $N_{\varepsilon}\in\mathbb N^*$ such that for every $n,m\in\mathbb N^*\cap [N_{\varepsilon},\infty[$,
\begin{equation}\label{convergence_subsequences_2}
|T/\psi(n)|^{\alpha\wedge 1/q},|T/\psi(m)|^{\alpha\wedge 1/q}
\leqslant
\frac{\varepsilon}{16M(N,R)(C_{\varphi} + K)}\wedge 1
\end{equation}
and
\begin{equation}\label{convergence_subsequences_3}
\|H^{\psi(n)} - H^{\psi(m)}\|_{\infty,T}\leqslant
\frac{\varepsilon}{16M(N,R)}.
\end{equation}
Consider $n,m\in\mathbb N^*\cap [N_{\varepsilon},\infty[$ and let $p$ be the least common multiple of $\psi(n)$ and $\psi(m)$. By Inequality (\ref{convergence_subsequences_1}):
\begin{small}
\begin{eqnarray*}
 \|Y^{\psi(n)} - Y^{\psi(m)}\|_{\infty,T}^{2} & \leqslant &
 2(\|Y^{\psi(n)} - Y^p\|_{\infty,T}^{2} +\|Y^{\psi(m)} - Y^p\|_{\infty,T}^{2})\\
 & \leqslant &
 8M(N,R)((C_{\varphi} + K)|T/\psi(n)|^{\alpha}
 +\|H^{\psi(n)} - H^{\psi(\psi^{-1}(p))}\|_{\infty,T})\\
 & &
 + 8M(N,R)((C_{\varphi} + K)|T/\psi(m)|^{\alpha}
 +\|H^{\psi(m)} - H^{\psi(\psi^{-1}(p))}\|_{\infty,T}).
\end{eqnarray*}
\end{small}
\newline
Since $p\geqslant\psi(n)$ and $p\geqslant\psi(m)$, $\psi^{-1}(p)\geqslant n\vee m\geqslant N_{\varepsilon}$. Then, by (\ref{convergence_subsequences_2}) and (\ref{convergence_subsequences_3}) together:
\begin{displaymath}
\|Y^{\psi(n)} - Y^{\psi(m)}\|_{\infty,T}^{2}
\leqslant\varepsilon.
\end{displaymath}
Therefore, $(Y^{\psi(n)})_{n\in\mathbb N^*}$ is a
uniformly converging sequence and by Equation (\ref{XHYn_relation}), $(X^{\psi(n)})_{n\in\mathbb N^*}$ also. In the sequel, the limit of $(Y^{\psi(n)})_{n\in\mathbb N^*}$ (resp. $(X^{\psi(n)})_{n\in\mathbb N^*}$) is denoted by $Y^*$ (resp. $X^*$).
\\
\\
On the other hand,
consider $(s,t)\in\Delta_T$, $z\in\cap_{\tau\in [s,t]} C(\tau)$ and $\tau\in [s,t]$. By Lemma \ref{inequality_Yn}:
\begin{displaymath}
\langle z - H^{\psi(n)}(\tau),Y^{\psi(n)}(t) - Y^{\psi(n)}(s)\rangle
\geqslant\frac{1}{2}(\|Y^{\psi(n)}(t)\|^2 -\|Y^{\psi(n)}(s)\|^2).
\end{displaymath}
So, when $n$ goes to infinity:
\begin{displaymath}
\langle z - H^*(\tau),Y^*(t) - Y^*(s)\rangle
\geqslant\frac{1}{2}(\|Y^*(t)\|^2 -\|Y^*(s)\|^2).
\end{displaymath}
Therefore, by Proposition \ref{characterization_sweeping_process}:
\begin{displaymath}
-\frac{dDY^*}{d|DY^*|}(t)
\in N_{C(t) - H^*(t)}(Y^*(t))
\textrm{ $|DY^*|$-a.e.}
\end{displaymath}
Moreover, since $(X^{\psi(n)})_{n\in\mathbb N^*}$ is a sequence of step functions uniformly converging to $X^*$, the definition of $(H^{\psi(n)})_{n\in\mathbb N^*}$ given by Equality (\ref{Hn_definition}) ensures that:
\begin{displaymath}
H^*(t) =\int_{0}^{t}f(X^*(s))ds + W(t)
\textrm{ $;$ }
\forall t\in [0,T].
\end{displaymath}
Since the solution $(X,Y)$ to (\ref{rough_sweeping_process}) is unique by Proposition \ref{uniqueness_RS}, $(X^*,Y^*) = (X,Y)$ and $H^* = X - Y$.
\\
\\
We have proved that, for each subsequence of $(X^n,Y^n)_{n\in\mathbb N}$, we can extract a further subsequence which converges uniformly to the solution $(X,Y)$. Thus $(X^n,Y^n)_{n\in\mathbb N^*}$ converges uniformly to $(X,Y)$.
\end{proof}
%


%
\section{Sweeping processes perturbed by a stochastic noise directed by a fBm}
First of all, let us recall the definition of fractional Brownian motion.
%


%
\begin{definition}\label{fBm}
Let $(\textrm B(t))_{t\in [0,T]}$ be a $d$-dimensional centered Gaussian process. It is a fractional Brownian motion of Hurst parameter $H\in ]0,1[$ if and only if,
\begin{displaymath}
\normalfont{\textrm{cov}}(B_i(s),B_j(s)) =
\frac{1}{2}(|t|^{2H} + |s|^{2H} - |t - s|^{2H})\delta_{i,j}
\end{displaymath}
for every $(i,j)\in\llbracket 1,d\rrbracket^2$ and $(s,t)\in [0,T]^2$.
\end{definition}
\noindent
Fore more details on fractional Brownian motion, we refer the reader to Nualart \cite[Chapter 5]{NUALART06}.
\\
\\
Let $B := (B(t))_{t\in [0,T]}$ be a $d$-dimensional fractional
Brownian motion of Hurst parameter $H\in ]1/3,1[$, defined on a
probability space $(\Omega,\mathcal A,\mathbb P)$.
\\
\\
By Garcia-Rodemich-Rumsey's lemma (see Nualart \cite[Lemma
A.3.1]{NUALART06}), the paths of $B$ are $\alpha$-H\"older
continuous for every $\alpha\in ]0,H[$. So, in particular, the paths of $B$ are continuous and of finite
$p$-variation for every $p\in ]1/H,\infty[$. By Friz and Victoir
\cite[Proposition 15.5 and Theorem 15.33]{FV10}, there exists an
enhanced Gaussian process $\mathbf B$ such that $\mathbf B^{(1)} = B$.
\\
\\
Consider $b\in C^{[p] + 1}(\mathbb R^e)$, $\sigma\in C^{[p] + 1}(\mathbb R^e,\mathcal M_{e,d}(\mathbb R))$ and the following sweeping process, perturbed by a pathwise stochastic noise directed by $\mathbf B$:
\begin{equation}\label{fractional_sweeping_process}
\left\{
\begin{array}{rcl}
 X(t) & = &
 H(t) + Y(t)\\
 H(t) & = &
 \displaystyle{
 \int_{0}^{t}b(X(s))ds +
 \int_{0}^{t}\sigma(X(s))d\mathbf B(s)}\\
 -\displaystyle{\frac{dDY}{d|DY|}}(t) & \in & N_{C_H(t)}(Y(t))
 \textrm{ $|DY|$-a.e. with }
 Y(0) = a.
\end{array}
\right.
\end{equation}
In the following, since $H$ can be deduced from $X$, and
$Y$ from $X$ and $H$, we say that $X$ is a solution to Problem
(\ref{fractional_sweeping_process}) if the corresponding triple
$(X,H,Y)$ satisfies (\ref{fractional_sweeping_process}).
\\
\\
Let $W := (W(t))_{t\in [0,T]}$ be the stochastic process defined by
\begin{displaymath}
W(t) :=
te_1 +\sum_{k = 1}^{d}B_k(t)e_{k + 1}
\textrm{ $;$ }
\forall t\in [0,T].
\end{displaymath}
By Friz and Victoir \cite[Theorem 9.26]{FV10}, there exists a $G\Omega_{p,T}(\mathbb R^{d + 1})$-valued
enhanced stochastic process $\mathbf W$ such that $\mathbf
W^{(1)} := W$. Consider also the map $f :\mathbb
R^e\rightarrow\mathcal M_{e,d + 1}(\mathbb R)$ defined by:
\begin{displaymath}
f(x)(u,v) :=
b(x)u +\sigma(x)v
\textrm{ $;$ }
\forall x\in\mathbb R^e
\textrm{$,$ }
\forall (u,v)\in\mathbb R^{d + 1}.
\end{displaymath}
So, Problem (\ref{fractional_sweeping_process}) can be reformulated as follow:
\begin{displaymath}
\left\{
\begin{array}{rcl}
 X(t) & = &
 H(t) + Y(t)\\
 H(t) & = &
 \displaystyle{
 \int_{0}^{t}f(X(s))d\mathbf W(s)}\\
 -\displaystyle{\frac{dDY}{d|DY|}}(t) & \in & N_{C_H(t)}(Y(t))
 \textrm{ $|DY|$-a.e. with }
 Y(0) = a.
\end{array}
\right.
\end{displaymath}
Therefore, the previous results of this paper apply to Problem
(\ref{fractional_sweeping_process}):

\begin{theorem}\label{theo:existence-fBm}
 (Existence) 
 Assume that, for every $t\in[0,T]$,
 $C(t)$ is a  random set with convex compact values with nonempty interior,
and that the paths of $C$ are continuous for the Hausdorff distance. 
Then 
Problem (\ref{fractional_sweeping_process}) has at least one solution, 
whose paths belong to $C^{p\normalfont{\textrm{-var}}}([0,T],\mathbb R^e)$,
for $p\in]1/H,\infty[$.
\end{theorem}
\begin{proof}
This is a direct pathwise application of Theorems
\ref{existence_Young} and \ref{existence_rough}. 
\end{proof}
\begin{proposition}\label{prop:uniqueness-fBm}
 (Existence and uniqueness for an additive fractional noise) Assume that, for every $t\in[0,T]$,
 $C(t)$ is a random set with convex compact values with nonempty interior,
and that the paths of $C$ are continuous for the Hausdorff distance. If $\sigma$ is a constant map,
then 
Problem (\ref{fractional_sweeping_process}) has a unique solution, 
whose paths belong to $C^{p\normalfont{\textrm{-var}}}([0,T],\mathbb R^e)$,
for $p\in]1/H,\infty[$.
\end{proposition}
\begin{proof}
This is a direct pathwise application of Theorem
\ref{existence_Young}, Theorem \ref{existence_rough} and Proposition \ref{uniqueness_RS}.
\end{proof}
\noindent
\textbf{Remark.} For instance, Proposition \ref{prop:uniqueness-fBm} ensures the existence and uniqueness of the solution to a multidimensional reflected fractional Ornstein-Uhlenbeck process.
%


%
\begin{proposition}\label{convergence_approximation_fBm}
Assume that, for every $t\in[0,T]$, $C(t)$ is a random set with convex compact values with nonempty interior,
and that the paths of $C$ are $\alpha$-H\"older continuous for the Hausdorff distance with $\alpha\in ]0,1[$. If $\sigma$ is a constant map, then the sequence of processes $(X^n)_{n\in\mathbb N^*}$ defined by
\begin{displaymath}
\left\{
\begin{array}{rcl}
 X_{0}^{n} & := & a\\
 X_{k + 1}^{n} & = & p_{C((k + 1)T/n)}(X_{k}^{n} + b(X_{k}^{n})T/n +\sigma B(kT/n,(k + 1)T/n))\textrm{ $;$ } k\in\llbracket 0,n - 1\rrbracket\\
 X^n(t) & := & X_{k}^{n}\textrm{ $;$ }t\in [kT/n,(k + 1)T/n[\textrm{$,$ }k\in\llbracket 0,n - 1\rrbracket
\end{array}\right.
\end{displaymath}
for every $n\in\mathbb N^*$ 
converges pathwise uniformly to the unique solution $X$ to Problem (\ref{fractional_sweeping_process}).
\end{proposition}
%


%
\begin{proof}
This is a direct pathwise application of Theorem \ref{convergence_subsequences}.
\end{proof}
%


%
%

\end{document}